\documentclass[12pt]{amsart}

\usepackage{enumerate, pdfsync,fullpage}
\usepackage{amssymb,amsmath,amsthm}
\usepackage{color}
\usepackage{graphicx}
\usepackage[bookmarksnumbered,colorlinks,plainpages,backref]{hyperref}

\DeclareMathOperator{\Res}{Res}

\def\m{\medskip}

\newcommand{\eps}{\varepsilon}

\newcommand{\R}{\mathbb R}

\newcommand{\N}{\mathbb N}
\newcommand{\C}{\mathbb C}

\DeclareMathOperator{\Imm}{Im}

\DeclareMathOperator{\Rre}{Re}

\newcommand{\p}{\partial}

\newcommand{\z}{\bar z}

\newcommand{\dbarb}{\bar\partial_b}
\newcommand{\dbarbs}{\bar\partial^*_b}

\newcommand{\Boxb}{\Box_b}

\newcommand{\atopp}[2]{\genfrac{}{}{0pt}{2}{#1}{#2}}
\newcommand{\nn}{\nonumber}
\newcommand{\ep}{\epsilon}
\newcommand{\I}{\mathcal{I}}

\newcommand{\la}{\langle}
\newcommand{\ra}{\rangle}

\newtheorem{thm}{Theorem}[section]
\newtheorem{prop}[thm]{Proposition}
\newtheorem{lem}[thm]{Lemma}

\newtheorem{cor}[thm]{Corollary}

\newtheorem*{theorem*}{Theorem}
\theoremstyle{definition}
\newtheorem{defn}[thm]{Definition}

\newtheorem{example}[thm]{Example}

\theoremstyle{remark}
\newtheorem{rem}[thm]{Remark}

\newcommand{\hatq}{\hat q}

\begin{document}

\title{The Fundamental Solution to $\Box_b$ on Quadric Manifolds -- Part 4. Nonzero Eigenvalues}

\author{Albert Boggess and Andrew Raich}%
\address{
School of Mathematical and Statistical Sciences\\
Arizona State University\\
Physical Sciences Building A-Wing Rm. 216\\
901 S. Palm Walk\\
Tempe, AZ 85287-1804 }
\address{
Department of Mathematical Sciences \\ 1 University of Arkansas \\ SCEN 327 \\ Fayetteville, AR 72701}

\email{boggess@asu.edu, araich@uark.edu}

\thanks{This work was supported by a grant from the Simons Foundation (707123, ASR)}%

\begin{abstract} This paper is the fourth of a multi-part series in which we study the geometric and analytic properties of the Kohn Laplacian 
and its inverse on general quadric submanifolds of $\mathbb{C}^n\times\mathbb{C}^m$. 
The goal of this article is explore the complex Green operator in the case that the eigenvalues of the directional
Levi forms are nonvanishing. We 1) investigate the geometric conditions on $M$ which the eigenvalue condition forces, 2) establish optimal pointwise
upper bounds on complex Green operator and its derivatives,
3) explore the $L^p$ and $L^p$-Sobolev mapping properties of the associated kernels, and
4) provide examples.
\end{abstract}

\maketitle

%
%
\section{Introduction}\label{sec:intro}
In this paper, we investigate {the} complex Green operator $N$ on quadric submanifolds $M \subset \C^n\times\C^m$ for which all the eigenvalues of the directional
Levi forms are nonzero. The complex Green operator is the (relative) inverse to the Kohn Laplacian $\Boxb$. By definition, a {\em quadric submanifold} is defined as
\begin{equation}
\label{eqn:M defn}
M = \{(z,w)\in\C^{n}\times\C^m : \Imm w = \phi(z,z)\}
\end{equation}
where $\phi:\C^{n}\times\C^{n} \to \C^m$ is a sesquilinear vector-valued quadratic form. The {\em Levi form in the direction of $\nu \in S^{m-1}$}, the unit sphere in $\R^m$, is defined
as $\phi_\nu (z,z) =\phi(z,z) \cdot \nu$. The {\em Kohn Laplacian} is defined as $\Boxb = \dbarb\dbarbs+\dbarbs\dbarb$ where 
$\dbarb$ is the usual tangential Cauchy-Riemann operator and $\dbarbs$ is its $L^2$ adjoint. The (relative) inverse to $\Boxb$ on $(p,q)$-forms, 
when it exists, is called the complex Green operator and denoted by $N_{p,q}$.
The existence of the complex Green operator produces the $L^2$-minimizing solution operator to the $\dbarb$-equation, 
$\dbarbs N_{p,q}$, in a canonical fashion. For background on the
$\dbarb$ and $\Boxb$-operators, please see \cite{Bog91,ChSh01, BiSt17}. 

In this paper, our main interest is the class of quadrics with  codimension $m\geq 2$
where the matrix associated to the scalar Levi form, $\phi_\nu (z,z)$ has only nonzero eigenvalues
for each $\nu\in S^{m-1}$.
We show, that the nonvanishing eigenvalue condition forces $n$ to be even (so replace $n$ with $2n$) with exactly half of the eigenvalues to be positive and half negative.
For $0 \leq q \leq 2n$, we establish sharp upper bounds on the size of 
$N_{0,q}$ 
and its derivatives in terms of the control geometry on $M$ that are analogous to the classical estimates on $N$ for the Heisenberg group or the finite type hypersurface type case (that is, $m=1$) in $\C^2$ \cite{NaRoStWa89,Chr91a,Chr91b, FeKo88a}. 
This allows us to invoke the theory of homogeneous groups to prove
$L^p$ and $L^p$-Sobolev mapping properties for appropriate derivatives of $N$.
When $q=n$,
$\Boxb$ is not solvable by \cite{PeRi03}, but we can still estimate the canonical relative fundamental solution for $\Boxb$ given by
$\int_0^\infty e^{-s\Boxb}(I-S_n)\, ds$ where $S_n$ is the orthogonal projection onto $\ker\Boxb$.
 We also provide several
examples, illustrating our estimates.

 More generally, when the eigenvalues are not bounded away from zero, the control distance fails to govern estimates on $N_{0,q}$.
This failure is apparent in some general hypersurface type CR manifolds as well as some simple higher codimension quadrics
\cite{Mac88,NaSt06,BoRa21III}. In higher codimension, the correct geometry is far from understood as the singularities of $N$ occur both on and off of the diagonal.

For a bit more background and history,
the tangential Cauchy-Riemann operator, or $\dbarb$, and the associated Kohn Laplacian $\Box_b$ are arguably the most important operators in several complex variables because they are intrinsically intertwined with the
complex geometry, topology, and analysis of CR manifolds.
Solving the $\Boxb$-equation is often a product of hard analysis and sophisticated functional analysis, and the solution produced by these techniques may have
excellent function theoretic properties but is not constructive (e.g., \cite{Sha85,Koh86,HaRa11,HaRa15,CoRa20}). Often, this approach is not (yet) sufficient to 
produce the estimates we seek on the solution in the higher codimension setting. Hence we restrict to the class of 
quadrics, which have a Lie group structure which helps provide a more explicit 
formula for the solution that is suitable to estimate.

In our opinion, one of the most beautiful results is the computation of $N_{p,q}$
on the Heisenberg group by Folland and Stein \cite{FoSt74p}. The problem, though, is that their technique does not easily generalize, especially to higher codimension. Consequently,
one of main approaches to the $\Boxb$-problem on these manifolds is through the $\Boxb$-heat equation. The first results in this direction
were for the sub-Laplacian on the Heisenberg group by Hulanicki \cite{Hul76} and Gaveau \cite{Gav77}. More results followed for
$\Boxb$ on quadrics of increasing generality \cite{BoRa09, YaZh08,CaChTi06,BeGaGr96,BeGaGr00, Eld09} culminating (so far) with
our paper \cite{BoRa11} where we compute the $\Boxb$-heat kernel on a general quadric. 
Virtually all of these results rely on the fact that we can identify $M$ with its tangent space
at the orgin, $\C^n\times\R^m$, and push the problem forward onto $\C^n\times\R^m$. 
The problem with these papers (ours included) is that if we put coordinates $(z,t)$ on $M$, the solution is only given up to a 
partial Fourier transform in $t$. Given that \cite{FoSt74p} is the gold standard (for us), we are taking the formula from \cite{BoRa11} and trying
to undo the Fourier transform and integrate out the time variable. This allows us to recover to both the projection onto $\ker\Boxb$ as well as $N_{0,q}$. 
In the earlier parts of the series, \cite{BoRa13q,BoRa21I,BoRa20II,BoRa21III}, we started with the formula for $\Boxb$-heat kernel and generated an
integral formula for both diagonal part of the complex Green operator as well as the projection onto $\ker\Boxb$. 
We also categorized the class of
quadrics of codimension 2 in $\C^4$ into three $\Box_b$-invariant groups and computed $0$th order asymptotics for the kernels for each of these groups. We noticed that in one case, 
where the directional Leviform has nonvanishing eigenvalues,
the complex Green operator was both
solvable and hypoelliptic. Additionally, the estimates were particularly good, allowing us to prove continuity results in $L^p$-Sobolev spaces, $1 < p < \infty$. In many respects, the current paper is 
a generalization of this case. 

In addition to our series of papers, Mendoza proves the following:
Let $M$ be a CR manifold of CR codimension $> 1$ whose Levi form is everywhere nondegenerate. 
Then $\Boxb$ computed with respect to any Hermitian metric
is hypoelliptic in all degrees except those corresponding to the number of positive or negative eigenvalues of the Levi form \cite{Men22}.
Additionally, in the special case that $\phi(z,z)$ is a sum of squares, Nagel, Ricci, and Stein \cite{NaRiSt01} proved pointwise upper bounds on 
both the complex Green operator and
the projections onto $\ker\Boxb$, and they established the $L^p$ theory in addition.

The outline of the paper is as follows.
In the next section, we state our main results, primarily Theorem \ref{thm:pointwise bounds}. We continue in Section \ref{sec:notation} where we define our notation and explore the 
geometric consequences of our hypotheses.  
The proof of Theorem \ref{thm:pointwise bounds} for $q \neq n$ is spread over Sections \ref{sec:review/formula for N} - \ref{sec:z large}.
In Section \ref{sec:q=n}, we discuss the adjustments to adapt the argument for the $q=n$ case.
We conclude the paper with several new examples in Section \ref{sec:examples}.

\section{Main Results}\label{sec:Main Results} 
Define the projection $\pi:\C^n\times\C^m \to \C^n\times\R^m$ by $\pi(z,t+is)=(z,t)$. For each quadric $M \subset\C^n\times\C^m$, the projection 
$\pi$ induces a CR structure and Lie group structure on $\C^n\times\R^m$, and we call this Lie group $G$ (or $G_M$). The projection is therefore
a CR isomorphism and we use the same notation for objects on $M$ and their pushfowards/pullbacks on $G$.

We introduce only the notation necessary to state the main results. 
Define the \emph{norm function} $\rho:\C^{2n}\times\R^m \to [0,\infty)$ by 
\[
\rho(z,t) = \max\{|z|,|t|^{1/2}\}  \approx |z|+|t|^{1/2}.
\]
For a multiindex $I = (I^1,I^2) \in \N_0^{4n+m}$, the multiindex $I^1 \in \N_0^{4n}$ records the differentiation in the
$z$ and $\z$-variables, and $I^2 \in \N_0^m$ records the $t$-derivatives. Given such a multiindex $I$, define the \emph{weighted order} of $I$ by
$\la I \ra = |I^1|+2|I^2|$ and the \emph{order} of $I$ by $|I| = |I_1|+|I_2|$. 
\begin{thm}\label{thm:pointwise bounds}
Let $M\subset\C^{2n}\times\C^m$ be a quadric submanifold defined by \eqref{eqn:M defn} with associated projection $G$,
and assume that eigenvalues of the directional Levi forms are nonzero.  
Let $0 \leq q \leq 2n$ and $N=N_{0,q}$.
For any multiindex $I\in\N_0^{4n+m}$, there exists a constant $C_I>0$ so that
\[
|D^I N(z,t)| \leq \frac{C_I}{\rho(z,t)^{2(2n+m-1)+\la I\ra}}.
\]
\end{thm}

\begin{rem} 
\begin{enumerate}[1.]
\item The homogeneous dimension of $M$ is $2(2n+m)$, and we are inverting an order two operator (with respect to $\rho$). This explains the power
of $\rho$ in the denominator of Theorem \ref{thm:pointwise bounds}.
\item The case $q=n$ is special because $\ker\Boxb \neq 0$. The relative fundamental solution that we estimate is 
$\int_0^\infty e^{-s\Boxb}(I-S_n)\, ds$ where $S_n:L^2_{0,n}(M)\to \ker\Boxb \cap L^2_{0,n}(M)$ is the orthogonal projection.
\end{enumerate}
\end{rem}

Let $W^{k,p}(M)$ denote the Sobolev space of forms on $M$ with $z$, $\bar{z}$ and $t$ derivatives of order $k$ are in $L^p(M)$.

\begin{thm}\label{thm:L^p}
Let $M\subset\C^{2n}\times\C^m$ be a quadric submanifold defined by \eqref{eqn:M defn} with associated projection $G$,
and assume that eigenvalues of the directional Levi forms are nonzero.  
Let $0 \leq q \leq 2n$ and $N=N_{0,q}$.
Given  
a multiindex $I\in\N_0^{4n+m}$ so that $\la I \ra =2$,
the operator $D^I N_{0,q}$ is exactly regular on $W^{k,p}(M)$ for all $k\geq 0$ and all $1 < p < \infty$. In other words,
$D^I N_{0,q}$ extends to a bounded operator on $W^{k,p}(M)$. In particular, $D^I N_{0,q}$ is a hypoelliptic operator.
\end{thm}

\begin{proof} The proof follows easily following the approach of \cite[Section 3]{BoRa20II}. Identifying $M$ with $\C^{2n}\times\R^m$, we can view
$M$ as a homogeneous group with norm function $\rho(z,t)$. From Theorem \ref{thm:pointwise bounds}, it follows that the integration kernel of $D^I N_{0,q}$ and its derivatives
have the appropriate pointwise decay. A second consequence of Theorem \ref{thm:pointwise bounds} is that $D^I N_{0,q}$ is a tempered distribution, and combining
this fact with the natural dilation structure and that $D^I N_{0,q}$ is a convolution operator shows that $D^I N_{0,q}$ is uniformly bounded on normalized bump functions. This is exactly
what is required to establish the $L^p$ boundedness, $1 < p < \infty$. From the fact that $D^I N_{0,q}$ is a convolution operator, boundedness on $W^{k,p}(\C^n\times\R^m)$
follows immediately.
\end{proof}

%
%
\section{Notation and Hypotheses}\label{sec:notation}
Suppose that $M$ is the quadric submanifold
\[
M = \{ (z,w) \in \C^n \times \C^m: \Imm w =  \phi (z,z) \}.
\]
Recall that for $\nu \in S^{m-1}$, $\phi_\nu (z,z) = \phi(z,z) \cdot \nu = z^* A_\nu z$ where $A_\nu $ is a 
Hermitian symmetric matrix.

\begin{prop}
If $m \geq 2$ and if the eigenvalues of $A_\nu$ are all nonzero for each $\nu \in S^{m-1}$, then $n$ must be even. Furthermore for each $\nu \in S^{m-1}$, half of the eigenvalues of $A_\nu$ are positive and half of the eigenvalues are negative,  counting multiplicity.
\end{prop}

\begin{proof}   
Note that if $\lambda $ is an eigenvalue for $A_\nu$, then $-\lambda$ is an eigenvalue for $A_{-\nu}$.  If $n$ is odd, then 
$\det A_{- \nu} =-\det A_\nu$. If $m \geq 2$, this change of sign in the determinant means that $\det A_{\nu'} =0$ for some other $\nu' \in S^{m-1} $.  Therefore, the assumption that all of eigenvalues are nonzero for each $\nu \in S^{m-1}$  implies that $n$ must be even. 

Also note that all the eigenvalues of $A_\nu$ are real. Let $p_\nu (\lambda) = \det(A_\nu-\lambda I)$ be the characteristic polynomial for $A_\nu$. Let $P_\nu$ be the set of the positive roots of $p_\nu$. We are assuming that $P_\nu$ is bounded away from zero for all $\nu \in S^{m-1}$. Let $K$ be a compact disc in the open right half plane which contains $P_\nu$ in its interior for all $\nu \in S^{m-1}$. The number of roots in $P_\nu$ is given by the Argument Principle:
\[
\textrm{Number of positive roots of $p_\nu$} = \frac{1}{2 \pi i} 
\oint_{\partial K} \frac{p_\nu'( \lambda) \, d \lambda}{p_\nu(\lambda)}.
\]
This is clearly a continuous integer-valued function of $\nu \in S^{m-1}$ which is a connected set for $m \geq 2$. Therefore,  the number of positive  roots of $p_\nu$ is constant for all $\nu \in S^{m-1}$.
Since $n$ is even and $A_{-\nu} = - A_\nu$, we see that $p_{-\nu} (-\lambda) = p_\nu(\lambda)$. 
Therefore if the number of positive roots of $p_\nu$ is $k$, then the number of negative roots of 
$p_{-\nu} ( \cdot)$  is also $k$, which in turn implies that the number of positive roots of $p_{-\nu}$ is $n-k$. Since the number of positive roots is constant in $\nu$, we conclude that $k = n-k$, and hence $k=n/2$.
\end{proof}

\subsection{The complex Green operator}
As a consequence of the above discussion, we assume the following:

\begin{itemize}

\item For each $\nu \in S^{m-1}$, there are $n$ positive eigenvalues $\mu_j^\nu $ for $j$ in some index set $P^\nu$ of cardinality $n$ 
from the set $\{ 1, 2, \dots , 2n \}$ and $n$ negative eigenvalues $\mu_k^\nu$ for $k \in (P^\nu)^c$, the complement of $P^\nu$.  
\end{itemize}
\begin{rem}
Given that our eigenvalues stay bounded away from $0$ 
independently of $\nu\in S^{m-1}$, we may arrange the indices so that $P^\nu = P$ is independent of $\nu$. 
\end{rem}

Denote the set of increasing $q$-tuples by $\I_q = \{K = (k_1,\dots,k_q) \in \N^q : 1 \leq k_1 < k_2 < \cdots < k_q  \leq 2n\}$.
To write the fundamental solution for $\Boxb$ \cite{BoRa21I} applied to a $(0,q)$-form of the form $f_K\, d\z^K$ for a fixed $K\in\I_q$, we 
need to establish some notation. 
Fix $\lambda \in \R^{m-1}\setminus\{0\}$ and set $\nu = \frac{\lambda}{|\lambda|}\in S^{m-1}$. We write $z\in\C^n$ in terms of the
unit eigenvectors of $\phi^\lambda$ which means that $z_j^\lambda = z_j^\nu$ is given by
\[
z^\nu = Z(\nu,z) = U(\nu)^*\cdot z
\]
where $U(\nu)$ is the matrix whose columns are the eigenvectors, $v^\nu_k$, $1 \leq k \leq 2n$ of the directional Levi form $\phi^\nu$,  and $\cdot $ 
represents matrix multiplication with $z$ written as a column vector. Note that the corresponding orthonormal basis of $(0,1)$-covectors for this basis
is 
\[
d \bar Z_j(\nu,z), \ \ 1 \leq j \leq 2n, \ \ \textrm{where} \ \ 
d \bar Z(\nu,z) =  U(\nu)^T \cdot d \bar z
\]
where $d \bar z$ is written as a column vector of $(0,1)$-forms and the superscript $T$ stands
for transpose. 
Note that $z^\nu =  Z(\nu,z) $ depends smoothly on $z \in \C^n$ 
but only  \emph{locally integrable} as a function of $\nu \in S^{m-1}$ \cite{Rai11}.

For each $K \in \I_q$, we will need to express $d \bar z^K$, in terms of $d \bar Z(\nu,z)^L$ for $L \in \I_q$. We have
\begin{equation}
\label{C(K,L)}
d \bar z^K = \sum_{L\in\I_q}  \det (\bar U(\nu)_{K,L}) \, d \bar Z(\nu,z)^L
\end{equation}
where $\bar U(\nu)_{K,L}$ is the $q \times q$ minor $\bar U(\nu)$ comprised of elements in the rows $K$  and columns $L$.
Note that if $q=2n$, then the above sum only has one term and 
$\det(\bar U(\nu)_{K,K}) = 1$. In addition, $\I_0 = \emptyset$, so the sum \eqref{C(K,L)} does not appear.

Until Section \ref{sec:q=n}, we work under the assumption that $0 \leq q \leq 2n$ is fixed and $q \neq n$. 
From \cite{BoRa21I}, the the fundamental solution to $\Box_b$ on $(0,q)$-forms spanned by $d\z^K$ is given by 
convolution with the kernel
\begin{align}
\label{eqn:N_K}
N_K (z,t)
&= K_{n,m} \sum_{L \in \I_q}  
\int_{\nu \in S^{m-1}} \det(\bar U(\nu)_{K,L}) \, d \bar Z(\nu,z)^L   \\
&\times\int_{r=0}^1  \bigg(  \prod_{\atopp{j\in L^c\cap P}{j \in L\cap P^c}} \frac{r^{ |\mu_j^\nu |} |\mu_j^\nu|} {(1-r^{|\mu_j^\nu |})} 
\prod_{\atopp{k\in L\cap P}{k \in L^c\cap P^c}} \frac{|\mu_{k}^\nu|} {(1-r^{|\mu_{k}^\nu|} ) } \bigg)
\frac{1}{  (A(r,\nu,z)-i \nu \cdot t)^{2n+m-1}}  \frac{dr \, d \nu}{r} \nn
\end{align}
where $d \nu$ is surface measure on the unit sphere $S^{m-1}$, the dimensional constant
\begin{equation}\label{eqn:K nm}
K_{n,m} =\frac{4^{2n}(2n+m-2)!}{2(2 \pi)^{m+2n}},
\end{equation}
and
\[
A(r,\nu, z) =  \sum_{j=1}^{2n} |\mu_j^\nu| \left( 
\frac{1+ r^{|\mu_j^\nu|}}{1- r^{|\mu_j^\nu|}} \right) |z_j^\nu|^2 .
\]

Taking derivatives in $z_k$ or $t_\ell$ is relatively straight forward because $z$ only appears in $A(r,\nu,z)$ and $t$ only appears in the
$\nu\cdot t$ term. In particular, we compute that for $1 \leq k \leq 2n$, 
\begin{equation}\label{eqn: dA/dz}
\frac{\p}{\p z_k} A(r, \nu, z) = \sum_{j=1}^{2n} |\mu_j^\nu| \left( 
\frac{1+ r^{|\mu_j^\nu|}}{1- r^{|\mu_j^\nu|}} \right) U(\nu)^*_{j,k} \cdot \overline{Z_j (\nu,z)}.
\end{equation}
Similarly, $\frac{\p^2 A(r,\nu,z)}{\p z_{k_1}\p z_{k_2}}=0$ as are all third (and higher) order derivatives. Also,
\begin{equation}\label{eqn:d^2 A/dz^2}
\frac{\p}{\p z_{k_1} \p\z_{k_2}} A(r, \nu, z) = \sum_{j=1}^{2n} |\mu_j^\nu| \left( 
\frac{1+ r^{|\mu_j^\nu|}}{1- r^{|\mu_j^\nu|}} \right) U(\nu)^*_{j, k_1} \cdot \overline{U(\nu)^*_{j,k_2}}.
\end{equation}

A key fact which will be used later is the following: If $P(u)$ is a polynomial in $u \in \C$, then 
\begin{equation}
\label{eqn:key}
\sum_{j=1}^{2n} P(\mu_j^\nu) |Z_j (\nu, z)|^2
= z^* \cdot U(\nu) \cdot  P(D_\nu) \cdot U(\nu)^* \cdot z = z^* \cdot P(A_\nu) \cdot z
\end{equation}
where $D_\nu $ is the diagonal matrix with the eigenvalues of $A_\nu$ as its diagonal entries.
The importance of this equation is as follows. The right side is a quadratic expression in $z$ and $\bar z$ with coefficients that are polynomials in the coordinates of $\nu$ (since $A_\nu$ depends linearly on $\nu$).

\subsection{Derivative Notation}\label{subsec:derivative notation} We define a multiindex $I = (I^1,I^2) \in \N_0^{4n+m}$ where $I^1 \in \N_0^{4n}$ is  
multiindex that records the
$z$ and $\z$-derivatives and $I^2 \in \N_0^m$ records the $t$-derivatives. Recall that the weighted order of $I$ is
$\la I \ra = |I^1|+2|I^2|$ and the order of $I$ is $|I| = |I_1|+|I_2|$. 
Each derivative in a $t$-variable introduces a component of $\nu$ into the numerator and increases the power of $(A(r, \nu, z)-i \nu \cdot t)$ in the
denominator by $1$. A derivative in a $z$-variable is more complicated to write down -- either the power of $(A(r, \nu, z)-i \nu \cdot t)$
increases by one in the denominator and a component of $\nabla_z A(r,\nu,z)$ is introduced in the numerator or the denominator remains unchanged
and a term in the numerator changes from \eqref{eqn: dA/dz} to (\ref{eqn:d^2 A/dz^2}). We will not need a precise accounting of the constants but 
only the number of first and second derivatives of $A(r,\nu,z)$ that appear. We denote $\nabla_{z,\z} A$ to be the vector of first derivatives
with respect to both the $z$ and $\z$ derivatives and $\nabla^2_{z,\z}$ to denote all of the second order derivatives of $A$. By an abuse of
notation, we write
\begin{align*}
&D^I \Big\{\frac{1}{(A(r, \nu, z)-i \nu \cdot t)^{2n+m-1}}\Big\}
= c_{n,m,|I_2|} D^{I_1} \Big\{\frac{\nu^{I_2}}{(A(r, \nu, z)-i \nu \cdot t)^{2n+m-1+|I_2|}}\Big\} \\
&= \sum_{\atopp{(I_1',I_1'')}{ |I_1'|+2|I_1''|=|I_1|}} c_{n,m,I_1',I_1'',|I_2|} \frac{\nu^{I_2} (\nabla_{z,\z}A(r,\nu,z))^{I_1'}(\nabla^2_{z,\z}A(r,\nu,z))^{I_1''}}
{(A(r, \nu, z)-i \nu \cdot t)^{2n+m-1+|I_1'|+|I_1''|+|I_2|}}.
\end{align*}
where $|I_1'|$ is the number of first order derivatives in $z$ or $\bar z$ and where
$|I_1''|$ is the number of second order derivatives in $z$ and $\bar z$.
Note that $|I_1'|+2|I_1''|=|I_1|$ and not $|I_1'|+|I_1''|$.
For example, suppose that $I_1 = (2,1,0,\dots,0,0)$, which is two $z_1$ factors and one $\z_1$ factor. Then 
\[
(\nabla_{z,\z}A(r,\nu,z))^{I_1} = \Big(\frac{\p}{\p z_1} A(r, \nu, z)\Big)^2 \Big(\frac{\p}{\p \z_1} A(r, \nu, z)\Big).
\]
and  $|I_1'|=1$, $|I_1''| =1$ and $|I_1|=3$.
We analyze each piece of $D^I N$ separately and consequently, the integral to estimate is
\begin{multline}
\label{eqn:NI1I2}
N_{I_1',I_1'',I_2} (z,t) 
= \sum_{L \in \I_q}  \int_{\nu \in S^{m-1}} \det(\bar U(\nu)_{K,L}) \, d \bar Z(\nu,z)^L  \int_{r=0}^1 
\bigg( \prod_{\atopp{j\in L^c\cap P}{j \in L\cap P^c}} \frac{r^{ |\mu_j^\nu |} |\mu_j^\nu|} {(1-r^{|\mu_j^\nu |})} 
\prod_{\atopp{k\in L\cap P}{k \in L^c\cap P^c}} \frac{|\mu_{k}^\nu|} {(1-r^{|\mu_{k}^\nu|} ) } \bigg) \\
\times \frac{\nu^{I_2} (\nabla_{z,\bar z}A(r,\nu,z))^{I_1'}(\nabla^2_{z,\bar{z}}A(r,\nu,z))^{I_1''}}
{(A(r, \nu, z)-i \nu \cdot t)^{2n+m-1+|I_1'|+|I_1''|+|I_2|}} \frac{d \nu \, dr}{r} .
\end{multline}

%
%

\section{The Case when $|t| \geq |z|^2$, $q \neq n$ } 
\label{sec:review/formula for N}
The tricky case is $|t|>|z|^2$ and so we will factor out a $|t|^{2n+m-1+|I_1'|+|I_1''|+|I_2|}$ from the denominator and we will rotate $\nu$ coordinates via an orthogonal matrix $M_t$ chosen 
so that $M_t(t/|t|)$ is the unit vector in the $\nu_1$ direction (so in the new coordinates, $\nu \cdot t = \nu_1)$. We also set $\nu^t = M_t^{-1}\nu$ and
\[
\hatq=\frac{z}{|t|^{1/2}} \in \C^{2n}, \ \ \textrm{and} \ \ Q({\nu^t},\hatq) = \frac{Z({\nu^t},z)}{|t|^{1/2}}
= \frac{U(\nu^t)^* \cdot z}{|t|^{1/2}}.
\]

Note that $|Q({\nu^t}, \hatq)|^2 = |\hatq|^2$ since $U_{\nu^t}$ is unitary.

\m
Since $(\nabla_{z, \bar z} A(r,\nu^t,z))^{I_1'}$ contains a monomial in $z,  \bar z$ of degree $I_1'$,   we obtain 
\begin{align*}
N_{I_1',I_1'',I_2}(z,t) &= |t|^{-(2n+m-1+\frac 12|I_1'|+ |I_1''|+|I_2|)} N_{I_1',I_1'',I_2} (q) \\
&= |t|^{-(2n+m-1+\frac 12\la I\ra)} N_{I_1',I_1'',I_2} (\hatq)
\end{align*} 
where
\begin{align}
\label{eqn:NI1I2 q}
&N_{I_1',I_1'',I_2} (\hatq) \\
&=   \sum_{L \in \I_q}  \int_{\nu^t \in S^{m-1}} \int_{r=0}^1 \det(\bar U(\nu^t)_{K,L}) \, d \bar Z(\nu^t ,z)^L B_L(r, {\nu^t})
\frac{(\nu^t)^{I_2} (\nabla_{z,\bar z}A(r,\nu^t,\hatq))^{I_1'}(\nabla^2_{z,\bar z}A(r,\nu^t,\hatq))^{I_1''}}
{(A(r,\nu^t, \hatq)-i \nu_1)^{2n+m-1+|I_1'|+|I_1''|+|I_2|}}\frac{d \nu \, dr}{r}  \nn
\end{align}
and
\begin{align}
\label{eqn:B_L}
B_L(r, \nu) &= 
\prod_{\atopp{j\in L^c\cap P}{j \in L\cap P^c}} \frac{r^{ |\mu_j^\nu |} |\mu_j^\nu|} {1-r^{|\mu_j^\nu |}} 
\prod_{\atopp{k\in L\cap P}{k \in L^c\cap P^c}} \frac{|\mu_{k}^\nu|} {1-r^{|\mu_{k}^\nu|}} \\
 \label{eqn:A-def}
A(r, \nu, \hatq) &= \sum_{j=1}^{2n} |\mu_j^\nu| \left( 
\frac{1+ r^{|\mu_j^\nu|}}{1- r^{|\mu_j^\nu|}} \right) |Q_j (\nu, \hatq)|^2 .
\end{align}

To prove Theorem \ref{thm:pointwise bounds} in the case that $|t| \geq |z|^2$ and $q \neq n$, it suffices to prove the following theorem.
\begin{thm}
  \label{thm:N I_1 I_2 mainestimate}
  There is a uniform constant $C>0$ so that $|N_{I_1',I_1'',I_2}(\hatq)| \leq C$
for all $\hatq \in \C^{2n}$.
\end{thm}

There are two primary terms which need to be analyzed: $B_L(r, \nu)$, and $A(r, \nu, \hatq)$. We first concentrate on the singularity at $r=1$. The singularity at $r=0$ is easier and is handled in Section \ref{sec:lower half}.

%
%
%
\section{Analysis of $B_L(r,\nu)$ in the case $r>1/2$, $q \neq n$} \label{sec:begin upper half}
It turns out that the key to analyzing $B_L(r,\nu)$ is $B_\emptyset(r,\nu)$.
To this end, 
for $0<r<1$ and $u \in \R$, let
\begin{equation}
\label{eqn:def f,g}
f(r,u) = \frac{ur^u}{(1-r^u)} \ \ \ g(r,u) = f(r,u) +u = \frac{u}{(1-r^u)}.
\end{equation}
Note that $g(r,u) = f(r, -u)$. Since $\mu_j^\nu >0$ for $j \in P$ and $\mu_{k}^\nu <0$ for 
$k \in P^c$, we can write
\[
B(r,\nu) = B_{\emptyset}(r,\nu) = \prod_{j\in P} \frac{r^{ |\mu_j^\nu |} |\mu_j^\nu|} {(1-r^{|\mu_j^\nu |})} 
\prod_{k\in P^c} \frac{|\mu_{k}^\nu|} {(1-r^{|\mu_{k}^\nu|} ) }
\]
then
\begin{align}
\label{eqn:Br1}
B(r, \nu) \frac{dr}{r} &=\prod_{j \in P} f(r, \mu_j^\nu) \prod_{k \in P^c} g(r, -\mu_{k}^\nu) \frac{dr}{r} \\
\label{eqn:Br2}
&=\prod_{j=1}^{2n} f(r, \mu_j^\nu)  \frac{dr}{r} .
\end{align}
Both descriptions of this term are useful.
Note that the eigenvalues $\mu_j^\nu$ are not necessarily smooth in $\nu \in S^{m-1}$ (though they are continuous). However as the next lemma shows, $B(r, \nu)$ is real analytic in both $0 <r<1$ and in $\nu \in S^{m-1}$ and this uses the fact that the eigenvalues are bounded away from zero.

\begin{lem}
\label{lem:real analyticity of B}
The function $B(r, \nu) = \prod_{j \in P} f(r, \mu_j^\nu) \prod_{k \in P^c} g(r, -\mu_{k}^\nu)$ is real analytic in both $0 <r<1$ and in $\nu \in S^{m-1}$.
\end{lem}

\begin{proof} Using (\ref{eqn:Br1}), write 
\begin{align*}
B(r, \nu) &=  B^+(r, \nu)  \cdot  B^-(r,  \nu) \ \ \textrm{where} \\
B^+(r, \nu) &=\prod_{j \in P} f(r, \mu_j^\nu); \ \ \ B^-(r, \nu) =  \prod_{k \in P^c} g(r, -\mu_{k}^\nu) .
\end{align*}
It suffices to show that $\ln B^+(r, \nu) $ and $\ln B^-(r, \nu)$ are real analytic in $0 <r<1$ and in $\nu \in S^{m-1}$. We have 
\[
\ln B^+(r, \nu) = \sum_{j \in P} \ln \tilde f(r, \mu_j^\nu)
\]
where $\tilde f(r,z) = \frac{z r^{z}}{(1-r^z)}$ for $z=u+iv$.  Since $\tilde f(r, z)>0$ for $z=u>0$, $\ln (\tilde f(r,z)) $ is real analytic in $0 < r<1$ and complex analytic as a function of $z=u+iv$ in a neighborhood, $U \subset \C$ containing the set $\{ u+i0; \  u>0\}$. Note that by hypothesis, there is a compact set $K \subset \{u+i0; \  u>0 \}$ which contains all the positive eigenvalues $\mu_j^\nu$ for $j \in P$ and $\nu \in S^{m-1}$. Let $\gamma \in U$ be a smooth simple closed curve which contains $K$. Let $D(\nu,z) = \det (A_\nu -zI)$ where recall that $A_\nu$ is the Hermitian matrix for $\phi_\nu (z, z)$. The eigenvalues $\mu_j^\nu$, $j \in P$ are the roots of the analytic function $z \to D(\nu,z)$ that lie inside $\gamma$. By standard Residue theory, we have
\[
\ln B^+(r, \nu) = \sum_{j \in P} \ln \tilde f(r, \mu_j^\nu)= \frac{1}{2 \pi i} \oint_{z \in \gamma} \frac{\ln \tilde f(r, z) D'(\nu,z) \, dz}{D(\nu,z)}
\]
where $D'(\nu,z)$ refers to the $z$-derivative of $D(\nu,z)$. Now observe that the right side is real analytic in $\nu \in S^{m-1}$ since $\nu \to A_\nu$ is real analytic in $\nu$ (and $D(\nu,z) \not=0$
for $z \in \gamma$). The proof of the analyticity of $\ln B^-(r, \nu)$ is similar. This completes the proof of the lemma.
\end{proof}

We observe that
\begin{align*}
B_L(r,\nu) &= B(r,\nu) \prod_{j \in L\cap P^c} \frac{f(r,-\mu^\nu_j)}{f(r,\mu^\nu_j)} \prod_{k \in L\cap P} \frac{g(r,\mu^\nu_k)}{g(r,-\mu^\nu_k)} = B(r,\nu)\prod_{j\in L} r^{-\mu^\nu_j}. 
\end{align*}

We need the following piece of notation for the next lemma.
For $J\in\I_q$ and $(\ell_1,\dots,\ell_q)\in\N^q$, set 
$\ep^{(\ell_1,\dots,\ell_q)}_J = (-1)^{|\sigma|}$ if $\{\ell_1,\dots,\ell_q\} = J$ as sets and $|\sigma|$ is the length of the permutation that takes 
$(j_1,\dots,j_q)$ to $J$. Set
$\ep^{(\ell_1,\dots,\ell_q)}_J=0$ otherwise.

It may be the case the $B_L(r,\nu)$ is \emph{not} analytic, however, we have the following lemma. We also use the notation that if $M$ is a matrix and
$J,L\in \I_q$, the $M_{J,L}$ is the $q\times q$ minor of $M$ with entries $M_{j\ell}$, $j \in J$, $\ell \in L$.
\begin{lem}
\label{lem:real analyticity of B_L}
The function 
\[
\nu \mapsto  \sum_{L \in \I_q} \det(\bar U(\nu)_{K,L}) \, d \bar Z(\nu,z)^L \prod_{j\in L} r^{-\mu^\nu_j} 
\]
is real analytic in both $0 <r<1$ and in $\nu \in S^{m-1}$.  Moreover, 
\begin{equation}\label{eqn:sum for r^-A}
\sum_{L\in\I_q} \det(\bar U(\nu)_{K,L})\, d\bar Z^L(\nu,z) \prod_{j\in L} r^{-\mu_j^\nu}  
= \sum_{J\in\I_q} \det ([r^{-\bar A_\nu}]_{K,J})\, d\z^J.
\end{equation}
\end{lem}

\begin{rem} In view of the above expression for $B_L(r, \nu)$, we record the following equation for future reference
\begin{equation}
\label{eqBL}
\sum_{L\in\I_q} \det(\bar U(\nu)_{K,L})\, d\bar Z^L(\nu,z) B_L(r, \nu) =
\sum_{J\in\I_q} \det ([r^{-\bar A_\nu}]_{K,J}) B(r, \nu) \, d\z^J.
\end{equation}
which is real analytic in $0<r<1$, $\nu \in S^{m-1}$ in view of Lemma \ref{lem:real analyticity of B}.
\end{rem}

\m
\begin{proof} Once we show (\ref{eqn:sum for r^-A}), the analyticity statement follows immediately from the fact that $\bar A_\nu$ 
depends analytically on $\nu$ and therefore the matrix $r^{-\bar A_\nu}$ will also
depend analytically on $\nu$.
\m

First, we record two basic equations. Suppose $M$ is a $N \times N$ 
matrix with complex entries and consider $w=Mz$, where $w, z \in \C^N$. If $1 \leq q \leq N$ and $K \in\I_q$, then
\begin{equation}
\label{dwK} 
d \bar w^K = \sum_{J \in I_q} \det (\bar M_{K,J}) \, d \bar z^J
\end{equation}
This is easily established using standard multilinear algebra.

Second, conjugation by $U(\nu)$ diagonalizes the matrix $A_\nu$, and  diagonalizes $r^{-A_\nu}$. In particular,
\begin{equation}
\label{Rdiagonal}
 R^{-\mu^\nu} = U(\nu)^T r^{-\bar A_\nu}   \bar U(\nu) 
\end{equation}
where $R^{-\mu^\nu}$ is the $(2n)\times(2n)$ matrix with real entries, $r^{-\mu^\nu_j}$, on the diagonal and zeros off of the diagonal.

Now we start with the left side of (\ref{eqn:sum for r^-A}):
\begin{align*}
\sum_{L\in\I_q} \det(\bar U(\nu)_{K,L})\, d\bar Z^L(\nu,z) \prod_{j\in L} r^{-\mu_j^\nu}
&=\sum_{L\in\I_q} \det(\bar U(\nu)_{K,L}) \det (R^{-\mu^\nu}_{L,L}) \, d\bar Z^L(\nu,z) \\
&= \sum_{L\in\I_q} \det([\bar U(\nu) R^{-\mu^\nu}]_{K,L})\, d\bar Z^L(\nu,z)
\end{align*}
where the second equation uses the fact that $R^{-\mu^\nu}$ is a diagonal matrix. Now use (\ref{Rdiagonal})
and the fact that $\bar U(\nu) U(\nu)^T = I$ to conclude that 
\begin{align*}
\textrm{Left side of (\ref{eqn:sum for r^-A})} &= \sum_{L \in \I_q}\det  [r^{-\bar A_\nu}  \bar U(\nu)]_{K,L} \,  d\bar Z^L(\nu,z) 
= d( r^{-\bar A_\nu} \bar z)^K
\end{align*}
where the last equality uses the equation $ z =  U(\nu)  Z(\nu,z)$ as well as (\ref{dwK}) with $w= r^{-A_\nu} z$. 
Now (\ref{eqn:sum for r^-A}) follows by using  (\ref{dwK}) to expand out the right side of the above equation
in terms of $d \bar z^J$.
\end{proof}

We make the following change of variables for $s>1$:
\begin{equation}
\label{eqn:change}
r=r(s)= \frac{s-1}{s+1} \ \ \textrm{or equivalently} \ \ s=\frac{r+1}{1-r} \ \ \textrm{with} \ \ 
\frac{dr}{r} = \frac{2 \, ds}{(s^2-1)}.
\end{equation} 
Note that $1/2 \leq r <1$ transforms to $s\geq 3$.

Our goal for the remainder of the section is to prove the following proposition.
\begin{prop}\label{prop: structure of non-z terms in terms of nu}
\begin{enumerate}[1.]
\item The expansion of $\frac{B(r(s), \nu) r'(s)}{r(s)}$ around $s=\infty$ is
\begin{align}
\label{eqn:B1}
\frac{B(r(s), \nu) r'(s)}{r(s)}&=
\frac{2}{2^{2n}(1-1/s^2)}\left[ \sum_{\ell=0}^{2n-1} P_\ell (\nu) s^{2n-\ell-2} +\frac{O(s, \nu)}{s^2} \right] \\
\label{eqn:B2}
\textrm{Typical Monomial in $P_\ell(\nu)$} &=\nu^{\ell - e}; \ \ \textrm{where $e$ is even with
$0 \leq e \leq \ell$}.
\end{align}
Here, $P_\ell (\nu) $ is a polynomial in $\nu = (\nu_1, \dots \, \nu_{m} )  \in S^{m-1}$ of total degree $\ell$.  By 
an abuse of notation, the term, $\nu^{\ell-e}$, in (\ref{eqn:B2})  stands for a monomial in the coordinates of $\nu$ of total degree $\ell-e$.

Additionally, the (Taylor) remainder $O(s, \nu) $ is real analytic in $s>1 $ and $ \nu \in S^{m-1}$. Furthermore  $O(s, \nu) $ is bounded in $s>1$.

\item Modulo coefficients (that are computable but not relevant to the estimate),
the expansion of $\det([r(s)^{-\bar A_\nu}]_{K,J})$ around $s=\infty$ is comprised of a sums of terms 
\begin{align}
\label{eqn:nu decomp of r^-A}
\frac{\nu^{\ell'-e'}}{s^{\ell'}}\  & \textrm{where $\ell' \geq 1$, $e'$ is an even integer with $0 \leq e' \leq \ell'$, and} \\
&\textrm{$\nu^{\ell'-e'}$ is a monomial of degree $\ell'-e'$ in the coordinates of $\nu \in S^{m-1}$}. \nn
\end{align}

\end{enumerate}
\end{prop}

To start the proof of Proposition \ref{prop: structure of non-z terms in terms of nu}, let
\[
F(s,u)=f(r(s),u), \ \  G(s,u) = g(r(s), u) .
\]
Using (\ref{eqn:Br2}), we obtain
\begin{equation}
\label{eqn:B-def}
\frac{B(r(s), \nu) r'(s)}{r(s)} = 2 \prod_{j=1}^{2n} F(s, \mu_j^{\nu}) \frac{1}{ (s^2-1)}.
\end{equation}

We will need to Taylor expand $B(r(s), \nu)$ in $s$ about $s=\infty$, which is equivalent to letting  $s=1/w$ and expanding about $w=0$. To this end, let
\begin{equation}
\tilde F(w, u) = w[ F(1/w, u) +u/2]= w \left[ g \left(\frac{1-w}{1+w}, u \right) - \frac{u}{2} \right]
\label{eqn:F(w,u) def}
\end{equation}

\begin{lem}
\label{lem: tilde F real analytic}
 $\tilde F(w,u) $ is a real analytic function of $w$ and $u$ for $-1 < w < 1 $ and $u \in \R$. In addition,
\begin{enumerate}[1.]

\item For each fixed $u$, the function $w \to \tilde F(w,u)$ is an even function of $w$;

\item For each fixed $w$, the function $u \to \tilde F(w,u)$ is an even function of $u$;

\item The coefficients in the Taylor series expansions of  $ \tilde F(w,u)$  in $w$ 
about $w=0$ are of the form:
\[
\textrm{$j$th coefficient} =  \left\{
\begin{array}{cc}
0 & \textrm{if $j$ is odd} \\
P_j(u) & \textrm{if $j$ is even} 
\end{array}
\right.
\]
where $P_j(u)$ is a polynomial of degree $j$ in $u$ that involves only even powers of $u$.

\end{enumerate}
\end{lem}

\begin{proof} We have
\begin{align}
\tilde F(w,u) &= \frac{wu}{1-\left( \frac{1-w}{1+w} \right)^u} -\frac{wu}{2} \nn\\
\label{eqn:exponential:eq}
&=  \frac{wu}{1-e^{u \ln \left( \frac{1-w}{1+w} \right)}}  -\frac{wu}{2}.
\end{align}
Since $1-e^z$ vanishes to first order in $z$ at the origin, the $(u,w)$ power series expansion of the denominator has a factor of $uw$, which cancels with the $uw$ in the numerator. The resulting term is analytic and nonvanishing in a neighborhood of the origin.
Hence $\tilde F$ is real analytic. Part (2) follows easily from \eqref{eqn:F(w,u) def}.
Parts (1) follows by a calculation (Maple helps). 
For Part (3), we expand the exponential term appearing in (\ref{eqn:exponential:eq}) and cancel the common factor of $uw$
to obtain
\[
\tilde F(w,u) = \left[
\frac{1}{L(w) + \frac{uw}{2!} L(w)^2 + \frac{(uw)^2}{3!} L(w)^3 + \dots} \right] -\frac{wu}{2}
\]
where $L(w) =w^{-1} \ln\left( \frac{1-w}{1+w} \right)$ is analytic on $-1 < w < 1$. From repeated $w$-differentiations of $\tilde F$, one can see that the $j$th $w$-derivative of $\tilde F$ at $w=0$ is a polynomial expression in $u$ of degree $j$. In view of Part (2), this expression is zero if $j$ is odd and only involves even powers of $u$ when $j$ is even as stated in Part (3). This concludes the proof of 
the lemma.
\end{proof}

We  let $w=1/s$ and unravel this lemma to imply the following expansions for 
$F(s, u)$.
\begin{align}
\label{eqn:F}
F(s,u)&=
\frac{s}{2} - \frac{u}{2} +\frac{u^2-1}{6s} - \frac{u^4 -5u^2 +4}{90s^3} + \sum_{j=3}^\infty
\frac{p_{2j} (u)}{s^{2j-1}} 
\end{align}
where $p_{2j}(u) $ is a polynomial in $u$ of degree $2j$ with only even powers of $u$. The above series converges uniformly on any closed subset of $\{ s >1 \}$. Note that $F$ has the linear term  $u/2$  and that all other terms involve only even powers of $u$. \m

Our next task is to use (\ref{eqn:F}) to expand the expression $B(r(s), \nu) $ given in (\ref{eqn:B-def})
in powers of $1/s$ (about $s=\infty$). To get started, here are the first few terms (in order of decreasing powers of $s$):
\begin{align}
\frac{B(r(s), \nu) r'(s)}{r(s)} &= \frac{2}{(s^2 -1)}\prod_{j=1}^{2n} F(s, \mu^\nu_j)\\
\label{eqn:B expansion}
&=\frac{2s^{2n}}{2^{2n}(s^2 -1)} \prod_{j=1}^{2n} \left[ 1 - \frac{\mu_j^\nu}{s} + \frac{(\mu_j^\nu)^2 -1}{3s^2} + \sum_{k=2}^\infty \frac{p_{2k} (\mu_j^\nu)}{s^{2k}} \right]
\end{align}
where $p_{2k} (u)$ is a polynomial of degree $2k$ with only even powers of $u$.
\m

Now, we expand the product on the right (denoted by Product) in terms of symmetric polynomials in the variables $\mu_1^\nu, \dots, \mu_{2n}^\nu$. First, a definition.

\begin{defn} A {\em symmetric polynomial} of degree $m$ on $\R^N$ is a polynomial $P$ of degree $m$ in the variables $(u_1, \dots u_N) \in \R^N$ such that $P(u_1, \dots u_N)  = P(u_{\sigma(1)}, \dots u_{\sigma(N)}) $ for all permutations $\sigma$ on $\{ 1, 2, \dots, N\}$.

An \emph{allowable} multiindex $\alpha = ( \alpha_1 , \dots , \alpha_N\big)$ is a nonincreasing $N$-tuple of nonnegative integers, that is,
integers $\alpha_j$, $1 \leq j \leq N$, satisfying
$\alpha_1 \geq \alpha_2 \geq \cdots \geq \alpha_N \geq 0$. Let $|\alpha|= \alpha_1 + \dots + \alpha_N$ and define
\[
S^\alpha (u_1, \dots , u_N) = \sum'_{i_1, \dots, i_N} u_{i_1}^{\alpha_1} \dots
u_{i_N}^{\alpha_N}
\]
where the sum is taken over all {\em distinct} indices $i_1, \dots, i_N$ each ranging from 1 to $N$.
\end{defn}

Note the prime over the sum emphasizes that the indices $i_j$ are distinct. 
Also for clarity, if the $2n$-tuple $\alpha$ ends with multiple zeros, we stop writing after the first zero. For example, we write
$S^{1,0}(\mu_1^\nu,\dots,\mu_{2n}^\nu)$ for $S^{1,0,\dots,0} (\mu_1^\nu,\dots,\mu_{2n}^\nu)$.
Clearly each
$S^\alpha (u)$ is a symmetric polynomial of degree $|\alpha|$. For a fixed $m>0$, the collection of  $S^\alpha (u)$ over all allowable multiindices $\alpha$ with $|\alpha| =m$
forms a basis of the space of symmetric polynomials of degree $m$ on $\R^N$.

From an examination of the product in (\ref{eqn:B expansion}) and using the fact that 
$p_{2k} (u)$ is a polynomial of degree $2k$ with only even powers of $u$, we obtain the 
following lemma.

\begin{lem} 
\label{lem: first expansion}
For $\ell \geq 0$, the coefficient of $\frac{1}{s^\ell}$ in the Product on the right side of (\ref{eqn:B expansion}) is a linear combination of 
\[
S^{\alpha} (\mu_1^\nu, \dots, \mu_{2n}^\nu), \ \ \textrm{with $|\alpha| = \ell, \ \ell -2, \  \ell -4, \dots , \ell -e$}
\]
where $e$ is the largest even integer which is less than or equal to $\ell$.
\end{lem}

As an illustration of this lemma, we write out the first few terms of the Product 
on the right side of (\ref{eqn:B expansion})
\begin{align*}
\textrm{Product} &=  1 - s^{-1}\sum_{k=1}^{2n} \mu^\nu_k + s^{-2}
\left((1/3) \sum_{k=1}^{2n}[ (\mu_k^\nu)^2 -1 ]  
+\sum_{j \not=k} \mu_j^\nu \mu_k^\nu \right) + \dots  \\
&= 1-\frac{1}{s} S^{1,0} (\mu^\nu)+ \frac{1}{s^2} \Big( (1/3) (S^{2,0} (\mu^\nu) -2n) + S^{1,1,0} (\mu^\nu) \Big) + \dots .
\end{align*}

\m

Now we need transform the $S^\alpha (\mu^\nu)$ into a more useful basis involving elementary symmetric functions.

\begin{defn} For $0 \leq \ell \leq N$, the elementary symmetric function of degree $\ell$ in $\R^N$ is
\begin{equation}
\label{def:elem}
E_\ell (u) = \sum_{(j_1,\dots,j_\ell)\in\I_\ell}  u_{j_1} \cdots  u_{j_\ell}
 .
\end{equation}
\end{defn}

With $N=2n$, the key fact about the $E_\ell(\mu^\nu)$ is that they appear as coefficients in the characteristic polynomial for $A_\nu$:
\begin{equation}
\label{symmetric:eq}
\det (A_\nu - \lambda I) =\lambda^{2n} + \sum_{\ell=1}^{2n} (-1)^\ell E_\ell (\mu^\nu) \lambda^{2n-\ell} .
\end{equation}

Note that each row of $A_\nu$ depends linearly and homogeneously on $\nu$ 
and thus the coefficient of $\lambda^{2n- \ell}$, i.e., 
$E_\ell (\mu^\nu)$, is a homogenous polynomial of degree $\ell$  in the coordinates of $\nu = (\nu_1, \dots, \nu_m) \in S^{m-1}$.  
As a consequence, we have

\begin{lem}
$E_\ell (\mu^\nu)$, is a homogenous polynomial of degree $\ell$  in the coordinates of $\nu = (\nu_1, \dots, \nu_m) \in S^{m-1}$.
\end{lem}

In particular, $E_\ell (\mu^\nu)$ is analytic in $\nu$ even though the eigenvalues $\mu_j^\nu$ are not necessarily differentiable in $\nu$.

\begin{defn}
Suppose $L= (\ell_1, \dots, \ell_j, \dots )$ is a multiindex (of indeterminate length) with 
$\ell_j \geq \ell_{j+1} $ and only a finite number of the $\ell_j$ are nonzero. For
$u=(u_1, \dots , u_N)$, define
\[
E^L(u) = E_{\ell_1} (u) \cdot E_{\ell_2} (u) \cdots E_{\ell_N} (u), \ \ 
\]
$E^L(u)$ is a symmetric polynomial of degree $|L| = \ell_1 + \dots + \ell_j + \dots$
\end{defn}

The next theorem is \cite[Theorem 7.4.4]{Sta99}. 

\begin{thm} 

For a given integer, $m \geq 1$, the collection of 
\[
\{E^L(u); \ |L|=m; \ u \in \R^N \}
\]
 is a basis for
the space of symmetric polynomials of degree $m$ on $\R^N$.

\end{thm}

The following corollary follows from this theorem and Lemma \ref{lem: first expansion}.

\begin{cor}
\label{cor:B-expansion in nu}
In the expansion of $B(r(s), \nu)\frac{r'(s)}{r(s)}$ given in (\ref{eqn:B expansion}), the coefficient of $s^{2n - 2- \ell}$ 
is expressible as a linear combination of 
\[
E^L(\mu^\nu) =  E_k(\mu^\nu)^{n_k} \cdots  E_2(\mu^\nu)^{n_2} E_1(\mu^\nu)^{n_1} \dots , \ \ k \geq 1
\]
where $L=(n_k, \dots, n_1)$ with 
 $|L|=n_1 + 2 n_2 + \dots k n_k = \ell -e$, where $e$ is an even integer with $0 \leq e \leq \ell$. Moreover, this coefficient is a linear combination of monomials
  in the components of $\nu =(\nu_1, \dots , \nu_m) \in S^{m-1}$ each having degree $\ell - e$.
  
\end{cor}

We will not need to know the exact values of the coefficients in this expansion. Rather, the key phrase is the last sentence in the above corollary: 
{\em the coefficient of $s^{2n-2- \ell}$ is a linear combination of monomials
 in the components of $\nu =(\nu_1, \dots , \nu_m) \in S^{m-1}$ each having degree $\ell - e$}.\m

\begin{proof}[Proof of Proposition \ref{prop: structure of non-z terms in terms of nu}.]
In view of Corollary \ref{cor:B-expansion in nu}  and (\ref{eqn:B expansion}), equations \eqref{eqn:B1} and \eqref{eqn:B2} both hold. Additionally, the 
real analyticity of the Taylor remainder term $O(s,\nu)$ for $s>1$ and $\nu\in S^{m-1}$ is assured from Lemma
\ref{lem:real analyticity of B}, the (Taylor) remainder $O(s, \nu) $. Furthermore  $O(s, \nu) $ is bounded in $s>1$.

The proof of Part 2 is simpler.
An expansion for $r(s)^{-u}$ about $s=\infty$ yields
\[
r(s)^{-u} = \Big(\frac{s-1}{s+1}\Big)^{-u}
= 1 - \frac {2u}s + \frac{2u^2}{s^2} - \frac{2u(1+2u^2)}{s^3} + \sum_{k=4}^\infty \frac{\tilde p_k(u)}{s^k}
\]
where $\tilde p_k(u)$ is a a polynomial that has only odd powers of $u$ if $k$ is odd and even powers of $u$ if $k$ is even
(this fact can be proven by setting $w = \frac 1s$, and Taylor expansion around $w=0$, and an induction argument on the form of the derivatives). This means
\begin{align*}
r(s)^{-A_\nu}
&= 1 - \frac {2A_\nu}s + \frac{2A_\nu^2}{s^2} - \frac{2A_\nu(1+2A_\nu^2)}{s^3} + \sum_{k=4}^\infty \frac{\tilde p_k(A_\nu)}{s^k}.
\end{align*}
Equation \eqref{eqn:nu decomp of r^-A} now follows from expanding the appropriate $q\times q$ minor determinant.
\end{proof}
  
\section{ Expansion of $A$ in Denominator in the case $1/2 \leq r <1$, $q \neq n$} \label{sec:A expansion, large r}
The formula for $A(r, \nu, \hatq)$ is given in (\ref{eqn:A-def}).
Using (\ref{eqn:def f,g}), the coefficient function in front of $|Q_j (\nu, \hatq)|^2$ is 
\[
f(r, \mu_j^\nu) + g(r, \mu_j^\nu) = 2f(r, \mu_j^\nu) + \mu_j^\nu 
\] 
which in the $s$ variables (where $r=r(s) = \frac{s-1}{s+1}$), using (\ref{eqn:F}), this becomes
\begin{equation}
\label{eqn:FP}
\textrm{Coefficient of $|Q_j|^2$} \ =2F(s, \mu_j^\nu) + \mu_j^\nu= s  + \sum_{k=1}^\infty
\frac{p_{2k} (\mu_j^\nu)}{s^{2k-1}} 
\end{equation}
where $p_{2k} (u) $ is a polynomial of degree $2k$ with only even powers of $u$. From  (\ref{eqn:F}),  the first two terms are
\[
p_2(u) =
\frac{u^2-1}{3}; \ \ p_4(u) = -\frac{u^4 -5u^2 +4}{45} \ \ \textrm{etc.}
\]
Now using  (\ref{eqn:A-def}),  (\ref{eqn:FP}), and (\ref{eqn:key}), we obtain
\begin{align}
\nn 
A(r(s), \nu, \hatq) -i\nu_1&=s|\hatq|^2+ \sum_{j=1}^{2n}  \sum_{k=1}^\infty \frac{p_{2k} (\mu_j^\nu)}{s^{2k-1}} |Q_j(\nu, \hatq )|^2 -i \nu_1 \\
\label{eqn:A-2-expansion}
&=( s|\hatq|^2 - i\nu_1)+\sum_{k=1}^\infty \frac{\hatq^* \cdot p_{2k} (A_\nu ) \cdot \hatq}{s^{2k-1}} .
\end{align}
We denote by $e_j \in \C^{2n}$ the $j$th unit vector $e_j = (0, \dots, 1, \dots, 0)$ (1 in the $j^{th}$ position).
We observe that 
\begin{equation}
\label{eqAzj}
\frac{\p A}{\p z_j} \Big|_{(r(s), \nu, \hatq)}= s\hatq^* \cdot e_j + \sum_{k=1}^\infty \frac{\hatq^* \cdot p_{2k} (A_\nu ) \cdot e_j}{s^{2k-1}}
\end{equation}
and
\begin{equation}
\label{eqAzzj}
\frac{\p^2 A}{\p z_{j_1}\p\z_{j_2}}\Big|_{(r(s), \nu, \hatq)} = se_{j_2}^* \cdot e_{j_1} + \sum_{k=1}^\infty \frac{e_{j_2}^* \cdot p_{2k} (A_\nu ) \cdot e_{j_1}}{s^{2k-1}}
\end{equation}
since the substitution of $\hatq$ for $z$ comes \emph{after} the differentiation.

Note the coefficients $ \hatq^* \cdot p_{2k} (A_\nu ) \cdot \hatq$ consist of quadratic terms in $\hatq$ and $\bar\hatq$ 
together with a linear combination of monomial terms in the coordinates of $\nu$ of degree $2k - e$ where $e$ is even with $0 \leq e \leq 2k$.

\m
%
%
\section{ Expanding the Kernel for $N$ in the case $1/2 \leq r <1$, $q \neq n$} \label{sec:N-section}
From (\ref{eqn:NI1I2 q}) and (\ref{eqBL}), to estimate $N_{I_1',I_1'',I_2}(\hatq)$, we must investigate the integrands 
\begin{equation}\label{eqn:N_{K,J} with s}
N_{K,J} (\hatq,s, \nu) =  \det ([r(s)^{-\bar A_\nu}]_{K,J}) \frac{B(r(s), {\nu^t}) r'(s)}{r(s)} 
\frac{(\nu^t)^{I_2} (\nabla_{z,\bar z}A(r(s),\nu^t,\hatq))^{I_1'}(\nabla^2_{z ,\bar z}A(r(s),\nu^t,\hatq))^{I_1''}}
{(A(r(s),\nu^t, \hatq)-i \nu_1)^{2n+m-1+|I_1'|+|I_1''|+|I_2|}}.
\end{equation}
For nonzero $V \in \C$, consider the Taylor expansion
\[
\frac{1}{(V+ \zeta)^{2n+m-1+|I_1'|+|I_1''|+|I_2|} } = \frac{1}{V^{2n+m-1+|I_1'|+|I_1''|+|I_2|}} + \sum_{j=1}^\infty 
\alpha_j \frac{\zeta^j}{V^{2n+m-1+|I_1'|+|I_1''|+|I_2|+j}}
\]
which converges uniformly for $|\zeta| \leq |V|/2$ (the values of $\alpha_j$ are unimportant).
We make use of the following expansions:
From (\ref{eqn:A-2-expansion}) with $V= s|q|^2 - i \nu_1$ and $\zeta = 
\sum_{k=1}^\infty \frac{\hatq^* \cdot p_{2k} (A_{\nu^t} ) \cdot \hatq}{s^{2k-1}} $ we have
\begin{align}
&\frac{(\nu^t)^{I_2} (\nabla_{z,\bar z}A(r(s),\nu^t,\hatq))^{I_1'}(\nabla^2_{z,\bar z}A(r(s),\nu^t,\hatq))^{I_1''}}
{  (A(r(s), \nu^t, \hatq)-i \nu_1 )^{2n+m-1+|I_1'|+|I_1''|+|I_2|}} \nn \\
& =(\nu^t)^{I_2}\bigg[ \frac{1}{( s|\hatq|^2 - i \nu_1)^{2n+m-1+|I_1'|+|I_1''|+|I_2|}} + \sum_{j=1}^\infty
\frac{\alpha_j \left[\sum_{k=1}^\infty \frac{\hatq^* \cdot p_{2k} (A_{\nu^t} ) \cdot \hatq}{s^{2k-1}} \right]^j}{( s|\hatq|^2 - i \nu_1)^{2n+m-1+j+|I_1'|+|I_1''|+|I_2|}}\bigg] 
\label{eqn:A-expansion, I_1,I_2}\\
&\times (\nabla_{z,\bar z}A(r(s),\nu^t,\hatq))^{I_1'}(\nabla^2_{z,\bar z}A(r(s),\nu^t,\hatq))^{I_1''}. \nn
\end{align}
Carefully writing out $(\nabla_{z,\bar z}A(r(s),\nu^t,\hatq))^{I_1'}$ and $(\nabla^2_{z,\bar z}A(r(s),\nu^t,\hatq))^{I_1''}$ would be more confusing than useful, as we only need the lead term and the generic expression
for the higher order terms. Using (\ref{eqAzj}), we write
\begin{align}
(\nabla_{z,\bar z}A(r(s),\nu^t,\hatq))^{I_1'}
&= s^{|I_1'|} C_{0,I_1'}\big((\hatq,\bar\hatq)^{|I_1'|}\big) \nn \\ 
&+ \sum_{K=1}^\infty\sum_{\atopp{k_1 + \dots + k_{|I_1'|}=K}{k_j\geq 0, \text{ all }j}} s^{|I_1'|-2K}C_{k_1,\dots,k_{I_1'},I_1'}\big((\hatq,\bar\hatq)^{|I_1'|}\big)
p_{2k_1,I_1'}(\nu^t) \cdots p_{2k_{|I_1'|},I_1'}(\nu^t)  \label{eqn: A' expansion}
\end{align}
and using (\ref{eqAzzj}), we have
\begin{align}\label{eqn: A'' expansion}
(\nabla_{z,\bar z}^2A(r(s),\nu^t,\hatq))^{I_1''}
&= s^{|I_1''|} C_{0,I_1''} 
+ \sum_{K=1}^\infty\sum_{\atopp{k_1 + \dots + k_{|I_1''|}=K}{k_j\geq 0, \text{ all }j}} s^{|I_1''|-2K}C_{k_1,\dots,k_{I_1''},I_1''}
p_{2k_1,I_1''}(\nu^t) \cdots p_{2k_{|I_1''|},I_1''}(\nu^t).
\end{align}
Here, $ C_{0,I_1'}\big((\hatq,\bar\hatq)^{|I_1'|}\big)$ and $C_{k_1,\dots,k_{I_1'},I_1'}\big((\hatq,\bar\hatq)^{|I_1'|}\big)$ denote polynomial expressions involving $\hatq $ and $\bar\hatq$ of degree $|I_1'|$
and $C_{k_1,\dots,k_{I_1''},I_1''}$ are constants (independent of $\hatq$).

From the derivative products \eqref{eqn: A' expansion} and \eqref{eqn: A'' expansion}, a typical term in 
$(\nabla_{z,\bar z}A)^{I_1'}(\nabla^2_{z,\bar z}A)^{I_1''}$ is of the form
\begin{equation}\label{eqn:derivative error}
s^{|I_1'|+|I_1''|-2K}C_{K}\big((\hatq,\bar\hatq)^{|I_1|}\big)
p_{2k_1}(\nu^t) \cdots p_{2k_{|I_1|}}(\nu^t) 
\end{equation}
where $C_{K}\big((\hatq,\bar\hatq)^{|I_1|}\big)$ is a polynomial in $\hatq$ and $\bar\hatq$ of degree at most $|I_1|$ and $k_1 + \dots + k_{|I_1'|+|I_1''|}=K\geq 1$ 
and each $k_j$ is a nonnegative integer. 

The main term is the lowest degree term in $\frac 1s$ and is given by
\[
(\nu^t)^{I_2} \frac{s^{|I_1'|+|I_1''|} C((\hatq, \bar\hatq)^{|I_1'|})}{( s|\hatq|^2 - i \nu_1)^{2n+m-1+|I_1'|+|I_1''|+|I_2|}} 
\]
where $C((\hatq, \bar\hatq)^{|I_1'|})$ is a monomial in terms in the coordinates for $(\hatq,\bar\hatq)$ of degree $|I_1'|$. Its exact expression is possible to compute but not relevant for this calculation.

Letting 
\[
K_{j,I_1',I_1''}= k_1 + \cdots +k_j +  k_{I_1}+ \cdots +  k_{|I_1'|+|I_1''|} \geq 1+ |I_1'|+|I_1''|, 
\]
a typical term from the expansion of \eqref{eqn:A-expansion, I_1,I_2} is
\begin{multline}
 \textrm{Typical Term in \eqref{eqn:A-expansion, I_1,I_2}}  \\
=(\nu^t)^{I_2}s^{|I_1'|+|I_1''|-2K_{j,I_1',I_1''}}C_{\tilde K}\big((\hatq,\bar\hatq)^{|I_1'|}\big)
p_{2\tilde k_1}(\nu^t) \cdots p_{2\tilde k_{|I_1'|+|I_1''|}}(\nu^t) \label{eqn:typical error term, I_1,I_2} \\
\times \frac{C(\hatq^j, \bar\hatq^j) \tilde p_{2k_1-e_1}(\nu^t) \dots \tilde p_{2k_j-e_j}(\nu^t)} 
{(s|\hatq|^2 -i\nu_1)^{2n+m-1+j+|I_1'|+|I_1''|+|I_2|} s^{(2k_1-1)+ \dots+(2k_j-1)}} 
\end{multline}
where $k_\ell\geq 1$ if $j\geq 1$ (and does not appear if $j=0$)
$C(\hatq^j ,\bar\hatq^j) $ stands for monomial terms in the coordinates for $\hatq$ of degree $j$ and $\bar\hatq$, of degree $j$, and where each 
$\tilde p_{2k_a -e_a} (\nu)$, $1 \leq a \leq j$ is a monomial in the coordinates of $\nu$ of degree $2k_a -e_a$. 
Here, $e_a$ is an even integer with $0 \leq e_a \leq 2k_a$.
Set $E_j= e_1+ \cdots + e_j$ and incorporate the matrix $M_t^{-1}$ into the $C(\hatq^j,\bar\hatq^j)$ term to obtain
\begin{equation}
\label{eqn:typical A}
\textrm{Typical Term in (\ref{eqn:typical error term, I_1,I_2})} =
\frac{C_t((\hatq,\bar\hatq)^{2j+|I_1'|} ) \nu^{2K_{j,I_1',I_1''}-E_j}\nu^{I_2}}{(s|\hatq|^2 -i\nu_1)^{2n+m-1+j+|I_1'|+|I_1''|+|I_2|} s^{2K_{j,I_1',I_1''}-(j+|I_1'|+|I_1''|)}}
\end{equation}
and $E_j$ is a an even integer with $0 \leq E_j \leq 2(k_1+\cdots+k_j)$. 
Note that we have used the same abuse of notation with $\nu^{2K_j-E_j}$ as we did in (\ref{eqn:B2}) and the dependence on $t$ is 
a (possibly nonsmooth but certainly bounded) dependence on $t/|t|$. 
We will not need all the terms in the expansion - just up through $j+|I_1'|+|I_1''|=2n-1$ with a remainder term involving $j+|I_1'|+|I_1''|=2n$ (and therefore
$K_{j,I_1',I_1''} := K_{2n}\geq 2n$).
In particular, using (\ref{eqn: A' expansion}) and (\ref{eqn: A'' expansion}), 
\begin{equation}
\label{eqn:remainder A}
\textrm{Typical Remainder Term in (\ref{eqn:A-expansion, I_1,I_2})} = \frac{O_t(1)O(\nu,s)} {(s|\hatq|^2- i\nu_1)^{4n+m-1} s^{2K_{2n}-2n}}
\end{equation}
where $O_t(1)$ is a real analytic function of the coordinates of $\hatq$ and $\bar\hatq$ that may depend on $t$.
Also, $O(\nu,s)$ stands for a real analytic function in $\nu \in S^{m-1}$ and $s>1$ and bounded in $s$. Note that the power of $s$ in the denominator is at least $2n$ since 
$K_{2n} \geq 2n$, as mentioned above.

 \m
 
In the expansions of $B(r(s), \nu) \frac{r'(s)}{r(s)}$ and $\det( [r^{-\bar A_\nu}]_{K,J})$ given in  (\ref{eqn:B1}) and \eqref{eqn:nu decomp of r^-A}, respectively,  writing
$(1-1/s^2)^{-1} = \sum_{j'=0}^\infty s^{-2j'}$. Therefore, by (\ref{eqn:B2}), we see that up to the coefficients of some polynomials,
a typical term in the expansion of $\det( [r^{-\bar A_\nu}]_{K,J}) \frac{ B(r(s), \nu) r'(s)}{r(s)}$ is
\begin{equation}
\label{eqn:typical B}
\textrm{Typical Term of }   \det( [r^{-\bar A_\nu}]_{K,J}) \frac{B(r(s), \nu) r'(s)}{r(s)} =
s^{2n-2j'-2-\ell'} \nu^{\ell'-e'}
\end{equation}
together with a remainder of the form $\frac{O(\nu,s)}{s^{2+2j'}}$. Note $e'$ is even and $0 \leq e' \leq \ell'$.
 
 Now the typical term of $N_{K,J}$ is the product of a term in (\ref{eqn:typical A}) with a term in  (\ref{eqn:typical B}). Therefore
 \begin{align}
 \label{eqn:typical N1}
 \textrm{Typical Term in $N_{K,J}$} &=
C(\hatq,  \bar\hatq)^{2j+|I_1'|}\frac{s^{N-2-\ell} \nu^{\ell -e}(\nu^t)^{I_2} }{(s |\hatq|^2-i \nu_1)^{N+m-1+|I_2|}} \ \ \ \textrm{where} \\
\label{Nj}
 N&=2n+j+|I_1'|+|I_1''|, \ \ \ell=\ell'+2j' +2K_{j+|I_1'|+|I_1''|}, \ \ e=e'+E_j+2j' .
\end{align}
Note that $e$ is even
 with $0 \leq e \leq \ell$, due to  the constraints listed in on the indices in (\ref{eqn:typical A}) and (\ref{eqn:typical B}).
 
 \m
 
 The remainder term for $N_{K,J}$ is the product of the remainders given in (\ref{eqn:remainder A}) and 
 the remainder given just after (\ref{eqn:typical B}): a typical term comprising the remainder is
 \begin{equation}
 \label{eqn:remainder N1}
 \textrm{Typical Remainder Term for $N_{K,J}$} = \frac{O(\hatq^{4n}) O (\nu, s)}{(s|\hatq|^2 -i \nu_1)^{4n+m-1} s^{2K_{2n}-2n+2+2j'} }
 \end{equation}
 where $O(\nu,s)$ is real analytic function in $\nu \in S^{m-1}$ and $s\geq 3$ and bounded in $s$. Note that the exponent in $s$ in the denominator is at least 
$2$ since $K_j \geq j, \ j \geq 1$.
 \m
 We will now show that the integral (over $\nu \in S^{m-1}$, and $s \geq 1$) of the typical term in 
 (\ref{eqn:typical N1}) is bounded in $\hatq$. We will also show the same for a remainder term in
 (\ref{eqn:remainder N1}).
 
  \m
 As to the first task, let $\hat r=|\hatq|^2 >0$ and define
  \[
  H_{N, \ell,m, e,I_2} (\hat r,s, \nu) = \frac{s^{N-2-\ell} \nu^{\ell+I_2-e}}{(s \hat r-i\nu_1)^{N+m-1+|I_2|}} .
  \]

To establish Theorem \ref{thm:N I_1 I_2 mainestimate} over the region $1/2 \leq r <1$, we need to show that for each $\ell \geq 0$, there is a uniform constant $C$ such that 
\begin{equation}
\label{eqn:basic estimate}
\Big| \int_{\nu \in S^{m-1} } \int_{s=3}^\infty H_{N, \ell,m, e,I_2} (\hat r,s, \nu) \, d \nu \Big| \leq C
\end{equation} 
for all $\hat r >0$ near zero.

\m

\section{ Unit Sphere Integrals.} \label{sec:sphere-integral}
To compute the integral of $ H_{N, \ell,m, e, I_2} (\hat r,s, \nu)$ over the unit sphere, $S^{m-1}$ in $\R^m$, we need to use some easy facts about spherical integrals:

\begin{enumerate}

\item Let $\nu=(\nu_1, \nu') \in S^{m-1}$, then $\nu' $ belongs to a $m-2$ dimensional sphere in $\R^{m-1}$ of radius
$|\nu'| = \sqrt{1- \nu_1^2}$.

\item Let $\theta $ be the ``angle'' between $\nu$ and the $\nu'$ plane; note that $\nu_1 = \sin (\theta)$, $-\pi/2 \leq \theta \leq \pi/2$; and $|\nu'| =\cos \theta$.

\item Surface measure on the unit sphere in $\R^m$ is $d \nu = (\cos \theta)^{m-2}\, d \theta\, d \nu'$
where $d \nu' $ is surface measure on, $S^{m-2}$,  the unit sphere in $R^{m-1}$.

\item The integral of any odd function of $\nu'$  over 
$S^{m-2}$,  the unit sphere in $R^{m-1}$, will be zero.

\end{enumerate}

Using the last fact, we claim that we can assume the monomial $\nu^{ \ell+I_2-e}$ depends on 
$\nu_1$ only. To see this, write $\nu^{ \ell+I_2-e} =(\nu')^a \nu_1^b$ with $|a|+b = |\ell|+|I_2|-|e|$. By (4),  if $|a|$ were odd,
then  the $\nu'$-integral would be zero. Thus we can assume $a=e'$ where $|e'|$ is even, which implies $|b|=|\ell|+|I_2|-(|e|+|e'|)=|\ell|+|I_2|-|E|$ with $|E|$ even. 
We can then factor out
the $(\nu')^a$ from the $\nu_1$ integral and we are left with  $\nu_1^{ \ell+I_2-E}$ 
within the $\nu_1$ integral.

We now change variables and let $x=\nu_1 = \sin \theta$, $-1 \leq x \leq 1$. Note that $\cos \theta = \sqrt{1-x^2}$
and $ d \theta = \frac{dx}{\sqrt{1-x^2}}$. Therefore
\begin{equation}
\label{surfmeasure}
d \nu = (1-x^2)^{(m-3)/2} \, dx \, d \nu'
\end{equation}
where $d \nu'$ is surface measure on $S^{m-2}$.
The desired estimate in (\ref{eqn:basic estimate})
will follow from the following lemma:

\begin{lem}

\label{lem:basicA-estimate}
For any nonnegative integers $N$, $m$ and $\ell$ with $m \geq 2$ and any even integer $E$ with $0 \leq E \leq |\ell|+|I_2|$, let
\[
A_{N,m}^{\ell, E,I_2} (\hat r) = \int_{x=-1}^1 \int_{s=3}^\infty 
\frac{(1-x^2)^{(m-3)/2} s^{N-2-\ell} x^{\ell -E+|I_2|} \, ds \, dx}
{(s\hat r-ix)^{N+m-1+|I_2|}}
\]
then $A_{N,m,I_2}^{\ell, E,I_2} (\hat r)$ is a smooth function of $\hat r>0$ up to $\hat r=0$.

\end{lem}

As shown in the proof, the lemma is not true if $E$ is odd.

\m

\begin{proof}[Proof of Lemma \ref{lem:basicA-estimate}] First write
\[
A_{N,m}^{\ell, E,I_2} (\hat r) = C_{N, \ell,I_2} D_{\hat r}^{N-(2+\ell)}  \left\{ B_{m,I_2}^{\ell, E}(\hat r) \right\}
\]
where $C_{N, \ell}$ is a constant and
\[
B_{m,I_2}^{\ell, E}(\hat r) = \int_{x=-1}^1 \int_{s=3}^\infty 
\frac{(1-x^2)^{(m-3)/2}  x^{\ell -E+|I_2|} \, ds \, dx}
{(s\hat r-ix)^{m+\ell +1+|I_2|}}.
\]
Here, $D_{\hat r}^j$ indicates the $j^{th}$ derivative with respect to $\hat r$. The index $j$ is allowed to be negative in which case this means the $|j|^{th}$ anti-derivative with respect to $\hat r$ (with a particular initial condition specified at a fixed value of $\hat r=\hat r_0>0$). 

The proof of the fact will be complete once we show $B_{m,I_2}^{\ell, E}(\hat r)$ is smooth for $\hat r>0$ up to $\hat r=0$. The $s$-integral can be computed to give:
\begin{equation}
\label{eqn:Bb}
B_{m,I_2}^{\ell, E}(\hat r) =\frac{1}{\hat r(m+\ell+|I_2|) }\,  b_{m,I_2}^{\ell, E} (\hat r)
\end{equation}
where
\[
b_{m,I_2}^{\ell, E}(\hat r) =\int_{x=-1}^1 
\frac{(1-x^2)^{(m-3)/2}  x^{\ell -E+|I_2|} \,  dx}
{(3\hat r-ix)^{m+\ell +|I_2|}}.
\]
We need to show $b_{m,I_2}^{\ell, E}(\hat r)$ is smooth in $\hat r>0$ up to $\hat r=0$ and that
\begin{equation}
\label{eqn:bm=0}
b_{m,I_2}^{\ell, E}(\hat r=0) =0
\end{equation}
for then (\ref{eqn:Bb}) will imply that $B_{m,I_2}^{\ell, E}(\hat r)$ is smooth in $\hat r>0$ up to $\hat r=0$.  To this end,
we note that for $\hat r>0$, the integrand of $b_{m,I_2}^{\ell, E}(\hat r)$ has an analytic extension in $x$ to the upper half of the complex plane. 
So we can deform the integral using Cauchy to the top half of the unit circle, denoted by $C^+$ from $z=-1$ to $z=+1$ to obtain
\[
b_{m,I_2}^{\ell, E} (\hat r)=\int_{z \in C^+}
\frac{(1-z^2)^{(m-3)/2}  z^{\ell -E+|I_2|} \,  dz}
{(3\hat r-iz)^{m+\ell+|I_2|}}.
\]
This expression shows that $b_{m,I_2}^{\ell, E}(\hat r)$ extends smoothly (in fact, analytically) in $\hat r$ to a 
neighborhood of $\hat r=0$. All that remains to show is that $ b_{m,I_2}^{\ell,E } (\hat r=0)=0$.
We have
\begin{equation}
\label{eqn:special eq1}
(-i)^{m+\ell+|I_2|} b_{m,I_2}^{\ell, E}(\hat r=0)=\int_{z \in C^+} 
\frac{(1-z^2)^{(m-3)/2}   \,  dz}{z^{m+E}}.
\end{equation}
If $m=3$, then this integral is $\int_{z \in C^+} \frac{dz}{z^{3+E}} =0$ since $e$ is even. If $m$ is odd and greater than $3$, then this integral can be reduced using integration by parts with $dv=1/z^{m+E} \, dz$
and  $u=(1-z^2)^{(m-3)/2}$ (note there are no boundary terms at $z=\pm 1$) to obtain
\[
b_m^{\ell, E} (\hat r=0)= c_{m, \ell} \int_{z \in C^+} 
\frac{(1-z^2)^{(m-5)/2}   \,  dz}{z^{m+E-2}}.
\]

 One can continue integrating by parts this until the power of $(1-z^2)$ is zero to obtain
\begin{equation}
\label{eqn:special eq2}
b_{m,I_2}^{\ell, E} (\hat r=0)= \tilde c_{m, \ell,I_2} \int_{z \in C^+} \frac{dz}{z^{3+E}} =0.
\end{equation}
This establishes (\ref{eqn:bm=0}) for $m$ odd.  (Note clearly, the above integral is {\em not} zero if $E$ is odd, which is why this assumption is so necessary).

If $m\geq 4$ is even, then we can integrate by parts until we obtain
\[
b_{m,I_2}^{\ell, E}(\hat r=0)= \tilde c_{m, \ell,I_2} \int_{z \in C^+} \frac{\sqrt{1-z^2} \,dz}{z^{4+E}}.
\]
Since $E$ is even, let $E=2k$ for a nonnegative integer $k$. Amazingly, there is a closed-form
antiderivative:
\begin{equation}
\label{eqn:special eq3}
\int \frac{\sqrt{1-z^2}}{z^{4+E}}  \,dz
= - \sum_{j=0}^k \frac{(1-z^2)^{j+3/2}}{z^{2j+3} }
\binom kj \frac{1}{(2j+3)}.
\end{equation}
Clearly this antiderivative vanishes at both $z=\pm 1$.  If $m=2$, then one can integrate by parts in (\ref{eqn:special eq1}) with $dv= \frac{z \, dz}{\sqrt{1-z^2}}$ and reduce to this integral to (\ref{eqn:special eq3}). Thus, Lemma \ref{lem:basicA-estimate}, and 
hence (\ref{eqn:basic estimate}) are proved.
\end{proof}

%
%
%
\section{The Remainder Term, $q \neq n$}
\label{sec:end upper half}

To restate the remainder in (\ref{eqn:remainder N1})
\[
\textrm{Remainder} = \frac{O (\nu, s)}{(s|\hatq|^2 -i \nu_1)^{4n+m-1} s^{J} }, \ \ \ 
\textrm{with } J \geq 2.
\]
We use the facts (1)-(3) about spherical integrals in the previous section with $x=\nu_1$.
Since $s^{-J}$ is integrable over $\{s \geq 3\}$ and since $O(\nu', \nu_1,x)$ is real analytic (and hence uniformly bounded) in $\nu' \in (\sqrt{1-x^2}) S^{m-2}$ ($m-2$ dimensional sphere of radius $\sqrt{1-x^2}$), it suffices to prove the following lemma,  which will finish the proof of Theorem \ref{thm:N I_1 I_2 mainestimate} for the integral over the region $1/2 \leq r <1$.

\begin{lem}
For $m \geq 2$, let
\[
R(s,\hat r,\nu') = \int_{x=-1}^1 \frac{(1-x^2)^{\frac{m-3}{2}}O(\nu', x,s) \, dx}{(s\hat r-ix)^{4n+m-1}}.
\]
Then $R(s,\hat r,\nu')$ is uniformly bounded for $s \geq 3$, $\hat r \geq 0$, and $\nu' \in (\sqrt{1-x^2}) S^{m-2}$.
\end{lem}

\begin{proof} Divide up the interval $-1 \leq x \leq 1$ into $\{|x| \geq 1/2\}$ and $\{|x | \leq 1/2\}$. The denominator is bounded below on $\{|x| \geq 1/2\}$. The numerator is also bounded except in the case $m=2$ in which case $(1-x^2)^{\frac{m-3}{2}}$ has an integrable singularity at $x= \pm 1$.

For the interval $\{|x | \leq 1/2\}$, we replace $x$ by $z \in \C$ and note that the integrand can be extended to analytic function $z$ in a complex neighborhood of the interval $-1/2 \leq x \leq 1/2$. Let $C$ be a path in this neighborhood and in the upper half plane which connects $z=-1/2$ to $z=1/2$ and otherwise does not intersect the real axis. Using Cauchy's Theorem, we have
\[
 \int_{x=-1/2}^{1/2} \frac{(1-x^2)^{\frac{m-3}{2}}O(\nu', x,s) \, dx}{(sr-ix)^{4n+m-1}}
 = \int_{z \in C} \frac{(1-z^2)^{\frac{m-3}{2}}O(\nu', z,s) \, dz}{(sr-iz)^{4n+m-1}}.
 \]
Since the denominator is uniformly bounded away from zero, for $z \in C$, $s \geq 3$ and $r=|\hatq|^2>0$, the integral on the right is uniformly bounded in $\nu'$, $r$, and $s$. This completes the proof.
\end{proof}

\section{ The case $0 \leq r \leq 1/2$, $q \neq n$}
\label{sec:lower half}
Our starting point is (\ref{eqn:NI1I2 q}) which equates to (\ref{eqn:N_{K,J} with s})
 but we wish to remain in the $r$ variable. 
We fix $K$, restrict the $r$ integral to $0 \leq r \leq 1/2$ and examine
\begin{align}
&N_{I_1',I_1'',I_2}^A(\hatq) \nn \\ 
&=   \sum_{L \in \I_q}  \int_{\nu \in S^{m-1}} \int_{r=0}^{\frac 12} \det(\bar U(\nu)_{K,L}) \, d \bar Z(z, \nu^t)^L B_L(r, {\nu^t})
\frac{(\nu^t)^{I_2} (\nabla_{z,\bar z}A(r,\nu^t,\hatq))^{I_1'}(\nabla^2_{z,\bar z}A(r,\nu^t,\hatq))^{I_1''}}
{(A(r,\nu^t, \hatq)-i \nu_1)^{2n+m-1+|I_1'|+|I_1''|+|I_2|}}\frac{d \nu \, dr}{r}  \label{eqn:NI1I2A q} \\
&= \sum_{J\in \I_q} d\z^J\bigg[ \int_{\nu \in S^{m-1}} \int_{r=0}^{\frac 12} \det ([r^{-\bar A_\nu}]_{K,J}) B(r, {\nu^t})
\frac{(\nu^t)^{I_2} (\nabla_{z,\bar z}A(r,\nu^t,\hatq))^{I_1'}(\nabla^2_{z,\bar z}A(r,\nu^t,\hatq))^{I_1''}}
{(A(r,\nu^t, \hatq)-i \nu_1)^{2n+m-1+|I_1'|+|I_1''|+|I_2|}} \bigg] \frac{d \nu \, dr}{r}. \label{eqn:NI1I2A q analytic form}
\end{align}
We denote by $N_{I_1',I_1'',I_2}^J(\hatq)$ the $d\z^J$ coefficient of $N_{I_1',I_1'',I_2}^A(\hatq)$.

We devote the remainder of this section to the proof of the following lemma, which will establish
Theorem \ref{thm:N I_1 I_2 mainestimate} for the integral over the region $0< r < 1/2$. 
 \begin{lem}
\label{lemma-est-lower}
\begin{equation}
\label{eqn:goal2}
|N_{I_1',I_1'',I_2}^J (\hatq)| \leq C \ \ \textrm{for all $\hatq = \frac{z}{|t|^{1/2}} \in \C^n$}
\end{equation}
where $C$ is a uniform constant.
\end{lem}

Recall from \eqref{eqn:B_L} that 
\[
B_L(r, \nu) = 
\prod_{\atopp{j\in L^c\cap P}{j \in L\cap P^c}} \frac{r^{ |\mu_j^\nu |} |\mu_j^\nu|} {1-r^{|\mu_j^\nu |}} 
\prod_{\atopp{k\in L\cap P}{k \in L^c\cap P^c}} \frac{|\mu_{k}^\nu|} {1-r^{|\mu_{k}^\nu|}}.
\]
Since $L \in \I_q$ and $q \neq n$, at least one of $L^c\cap P$ or $L\cap P^c$ is nonempty. This means there exist constants $C>0$ and 
\[
c_0 = \min\Big\{ \sum_{\atopp{j\in L^c\cap P}{j \in L\cap P^c}} |\mu_j^\nu| : \nu\in S^{m-1} \text{ and }L\in\I_q\Big\}
\]
so that
\begin{equation}
\label{eqn:B-est}
\Big| \frac{B(r, \nu)}{r} \Big| \leq  C  r^{c_0-1} \ \ \ \textrm{for $0<r<1/2$}.
\end{equation} 
From this estimate, it follows that the integrals in \eqref{eqn:NI1I2A q} and therefore in \eqref{eqn:NI1I2A q analytic form}
over $\{ 0 \leq r \leq 1/2\} \times \{ |\nu_1| \geq 1/2 \}$ are uniformly bounded for $\hatq \in \C^n$. 
Moreover, we know from (\ref{eqBL}) and the accompanying remark that the integrand of
$N_{I_1',I_1'',I_2}^J (\hatq)$ is real analytic in $\nu \in S^{m-1}$ and $0< r \leq 1/2$.

We now concentrate on the $\nu_1$-integral over $|\nu_1| \leq 1/2$.
We have
 \[
A(r, \nu, \hatq) = \sum_{j=1}^{2n} |\mu_j^\nu| \left( 
\frac{1+ r^{|\mu_j^\nu|}}{1- r^{|\mu_j^\nu|}} \right) |Q_j (\nu, \hatq)|^2 
\]
with 
\[
Q(\nu,\hatq) = |t|^{-1/2}Z(\nu,z) =  U(\nu)^* \cdot \hatq
\]
where $U(\nu)$ is the unitary matrix which diagonalizes the scalar valued Levi form in the normal direction $\nu$. For $u \in \R$, let 
\begin{equation}
\label{Lambda}
\Lambda(u) = |u| \left( \frac{1+r^{|u|}}{1-r^{|u|}} \right).
\end{equation}
As a generalization of (\ref{eqn:key}), we have
\begin{equation}
\label{eqn:Arnuq}
A(r, \nu ,\hatq) = \sum_{j=1}^{2n}  \Lambda(\mu_j) |Q_j (\nu, \hatq)|^2
= \hatq^* \cdot \Lambda(A_\nu) \cdot \hatq
\end{equation}
where $A_\nu$ is the Hessian matrix of $\Phi(z,z) \cdot \nu$. 
 Here
$\Lambda( A_\nu)$ is computed by replacing $|u|$ by $\sqrt{ A_\nu^2} $ in (\ref{Lambda}) and where $\left(I-r^{\sqrt{ A_\nu^2} }\right)^{-1}$ is  the matrix inverse of $I-r^{\sqrt{ A_\nu^2} }$.
Furthermore,
$r^{\sqrt{ A_\nu^2} }$ is defined as $\exp \left( \ln r \sqrt{ A^2_\nu}\right)$. Note that since all the eigenvalues of $A_\nu$ are real and bounded away from zero, 
the (operator) norm of the matrix  $r^{\sqrt{ A_\nu^2} }$,
for $0\leq r \leq 1/2$, is less than one since $\ln r<0$, guaranteeing the existence of the inverse of $I-r^{\sqrt{ A_\nu^2} }$. 
For this analysis to work, we need to know the map $\nu \to \sqrt{A_\nu^2}$ is analytic in $\nu$, established in the following lemma.

\begin{lem}
\label{lem:square root}
 The map $\nu \to  \sqrt{A_\nu^2}$ is analytic  for $\nu \in S^{m-1}$.
\end{lem}

\begin{proof}
Observe that the matrix $X=A_\nu^2$ is a Hermitian symmetric matrix with positive eigenvalues which are contained in a compact interval, say $ [c_0, R] \subset \R$ with $R>c_0>0$, for all $\nu \in S^{m-1}$. So consider the power series for $\sqrt{X}$ about $X=RI$:
\[
 \sqrt{X} = \sum_{n=0}^\infty  a_n [X-RI]^n
\]
which has radius of convergence  $R$. Since the open disc centered at $x=R$ of radius $R$ contains all the eigenvalues of $A_\nu^2$ in its interior, the following series converges uniformly in $\nu$:
\[
\sqrt{A_\nu^2} = \sum_{n=0}^\infty  a_n [A_\nu^2-RI]^n .
\]
This series is clearly analytic in $\nu \in S^{m-1}$.

\end{proof}

%

\begin{proof}[Proof of Lemma \ref{lemma-est-lower}]
The expression for $A$ given in (\ref{eqn:Arnuq}) and the discussion following shows that
\[
X(r, \nu, \hatq):=\left[ \frac{\partial}{\partial \nu_1} \{A(r,\nu^t,\hatq) \} -i I \right]^{-1}
\]
is a smooth matrix on $ \{0 < r \leq 1/2 \} \times \{ \nu \in S^{m-1} \}$ with $X(r, \nu, \hatq)  \to 0$
as $r \to 0$ (due to a factor of $\ln r$ in the denominator). Moreover, since $\frac{d r^u}{du} = r^u \ln r$, differentiation of $A(r,\nu^t,\hatq)$, $X(r,\nu,\hatq)$, 
$B(r,\nu^t)$, or  $\det ([r^{-\bar A_\nu}]_{K,J})$
produces a term of the same size with a possible additional $\ln r$ term. However, from \eqref{eqn:B-est}, there is always a $r^{c_0-1}$ term,
and $r^{c_0-1}|\ln r|^N$ is integrable at $0$ for any power $N$ since $c_0 >0$.

Now view the set $\{\nu \in S^{m-1}; \ |\nu_1| \leq 1/2 \} $ as the graph over the set
\[
\{ \nu=(\nu_1, \nu' ); \ |\nu_1| \leq 1/2;\  \nu'  \in S^{m-2} \}.
\]
We use integration by parts for the integral over $|\nu_1| \leq 1/2$ of (\ref{eqn:NI1I2A q analytic form}) as follows:
\begin{align*}
& \int_{\nu_1=-1/2}^{\nu_1 = 1/2}\frac{(\nu^t)^{I_2} (\nabla_{z,\bar z}A(r,\nu^t,\hatq))^{I_1'}(\nabla^2_{z,\bar z}A(r,\nu^t,\hatq))^{I_1''}}
{(A(r,\nu^t, \hatq)-i \nu_1)^{2n+m-1+|I_1'|+|I_1''|+|I_2|}}\frac{\det ([r^{-\bar A_\nu}]_{K,J})  B(r, \nu^t)}{r}
\, d \nu_1  \\
&=
 \int_{\nu_1=-1/2}^{\nu_1 = 1/2}  
\frac{\partial}{\partial \nu_1} \left\{ \frac{-(2n+m+|I_1'|+|I_1''|+|I_2|-2)^{-1} }
{(A(r, \nu^t, \hatq) - i \nu_1)^{2n+m+|I_1'|+|I_1''|+|I_2|-2}} \right\} \\
& \times \frac{(\nu^t)^{I_2}(\nabla_{z,\bar z}A(r,\nu^t,\hatq))^{I_1'}(\nabla^2_{z,\bar z}A(r,\nu^t,\hatq))^{I_1''}X(r, \nu^t, \hatq) 
\det ([r^{-\bar A_\nu}]_{K,J}) B(r, \nu^t)}{r} \, d \nu_1  \\
&= - \int_{\nu_1=-1/2}^{\nu_1 = 1/2}  
\frac{-(2n+m +|I_1'|+|I_1''|+|I_2|-2)^{-1}}{(A(r, \nu^t, \hatq) - i \nu_1)^{2n+m+|I_1'|+|I_1''|+|I_2|-2}} \\
&\times\frac{\partial}{\partial \nu_1}
\left\{
\frac{ (\nu^t)^{I_2} (\nabla_{z,\bar z}A(r,\nu^t,\hatq))^{I_1'}(\nabla^2_{z,\bar z}A(r,\nu^t,\hatq))^{I_1''}X(r, \nu^t, \hatq) 
\det ([r^{-\bar A_\nu}]_{K,J})B(r, \nu^t)}{r} \right\}  \, d \nu_1\\
&+ \textrm{Boundary Terms at $|\nu_1|=1/2$}.
\end{align*}
The power of $(A(r, \nu^t, \hatq) - i \nu_1)$ in the denominator has been reduced by one. As discussed above and analogous to (\ref{eqn:B-est}), we have 
\[
\bigg| \frac{\partial}{\partial \nu_1} \left\{
\frac{X(r, \nu^t, \hatq) \det ([r^{-\bar A_\nu}]_{K,J}) B(r, \nu^t) (\nu^t)^{I_2} (\nabla_{z,\bar z}A(r,\nu^t,\hatq))^{I_1'}(\nabla^2_{z,\bar z}A(r,\nu^t,\hatq))^{I_1''}}{r} \right\} \bigg| 
\leq  C  |\ln r|r^{c_0-1}
\]
which is integrable in $0 \leq r \leq 1/2$. The boundary terms are also controlled by a similar estimate.
 
 We continue integrating by parts in $\nu_1$ until we reduce the fractional expression involving
 $(A(r, \nu^t, \hatq) - i \nu_1)$ to a log-term, to obtain:
 \begin{align*}
& \int_{\nu_1=-1/2}^{\nu_1 = 1/2} \frac{(\nu^t)^{I_2} (\nabla_{z,\bar z}A(r,\nu^t,\hatq))^{I_1'}(\nabla^2_{z,\bar z}A(r,\nu^t,\hatq))^{I_1''}}{(A(r, \nu^t, \hatq) - i \nu_1)^{2n+m-1 +|I_1'|+|I_1''|+|I_2|}}
\frac{\det ([r^{-\bar A_\nu}]_{K,J}) B(r, \nu^t)}{r}
\, d \nu_1  \\
&= c_{n,m,I_1',I_1'',I_2} \int_{\nu_1=-1/2}^{\nu_1 = 1/2}  \, \ln \left[A(r, \nu^t, \hatq) - i \nu_1 \right]
\frac{ E[ X(r, \nu^t, \hatq),  \det ([r^{-\bar A_\nu}]_{K,J}),B(r, \nu^t),A(r,\nu^t,\hatq),I_1',I_1'',I_2]}{r}  \, d\nu_1 \\
&+ \textrm{Boundary Terms at $|\nu_1|=1/2$}&
\end{align*}
where $c_{n,m,I_1',I_1'',I_2}$ is a constant depending only on $n,m, I_1',I_1'',I_2$; 
$\ln$ is the principal branch of the logarithm defined on the right half-plane (note $A(r,\nu^t,\hatq) >0$); and the function
$E[ X(r, \nu^t, \hatq),  \det ([r^{-\bar A_\nu}]_{K,J}),B(r, \nu^t),A(r,\nu^t,\hatq),I_1',I_1'',I_2]$ is an expression involve a sum of products of $\nu_1$-derivatives of $X(r, \nu^t, \hatq)$,
$\det ([r^{-\bar A_\nu}]_{K,J})$, and $B(r, \nu^t)$ where the total number of derivatives is $2n+m+|I_1'|+|I_1''|+|I_2|-1$.
Note that $| \ln \left[A(r, \nu^t, \hatq) - i \nu_1 \right] |$ is integrable in $\nu_1$ uniformly in the other variables $\nu' \in S^{m-2}$ and $0 \leq r \leq 1/2$. In addition
\[
\Big| \frac{ E \left[ X(r, \nu^t, \hatq), \  B(r, \nu^t) \right]}{r} \Big| \leq 
C|\ln r|^{2n+m+|I_1'|+|I_1''|+|I_2|-1} r^{c_0-1}
\]
which is also integrable $\nu' \in S^{m-2}$ and $0 \leq r \leq 1/2$. Similar estimates hold for the boundary terms. 
This establishes (\ref{eqn:goal2}) and completes the proof of Lemma \ref{lemma-est-lower}. This also concludes the proof of Theorem  \ref{thm:N I_1 I_2 mainestimate} and hence establishes Theorem \ref{thm:pointwise bounds} when $|t| \geq |z|^2$.
\end{proof}

%
%
\section{The $|z|^2 \geq |t|$ case, $q \neq n$}\label{sec:z large}
\label{z large}

To complete the proof of Theorem   \ref{thm:pointwise bounds}, we have left to check the case when $|z|^2 \geq |t|$. As before, we break the integral up into
two cases: $0 < r \leq \frac 12$ and $\frac 12 < r \leq 1$.

\subsection{The case $\frac 12 < r \leq 1$.} As before, we start with the harder case. Fortunately, though, the bulk of the preliminary computations still hold. We take
\eqref{eqn:typical A} as our starting point. The differences between the $|t|$ and $|z|^2$ dominant cases, though, is that in the manipulations leading to \eqref{eqn:typical A}
we do not want to factor $|t|$ out of the integral and replace $z$ by $\hatq$. We also worry about the $C_t((\hatq,\bar\hatq)^{2j+|I_1'|}$ term which is now a
$C(z,\bar z)^{2j+|I_1'|}$ term. We will use size estimates and ignore completely the (uniformly bounded) $\nu$ terms. Thus, \eqref{eqn:typical A} simplifies to
\begin{equation}
|\textrm{Typical Term}| \leq
\frac{C(z,\bar z)^{2j+|I_1'|}}{(s|z|^2)^{2n+m-1+j+|I_1'|+|I_1''|+|I_2|} s^{2K_{j,I_1',I_1''}-(j+|I_1'|+|I_1''|)}}
\end{equation}
and we estimate
\begin{align*}
&\int_{s=3}^\infty \int_{\nu\in S^{m-1}}\frac{C(z,\bar z)^{2j+|I_1'|}}{(s|z|^2)^{2n+m-1+j+|I_1'|+|I_1''|+|I_2|} s^{2K_{j,I_1',I_1''}-(j+|I_1'|+|I_1''|)}}\, d\nu\, ds \\
&\leq C_{j,I_1',I_1'',I_2} \frac{|z|^{2j+|I_1'|}}{|z|^{2(2n+m-1+j+|I_1'|+|I_1''|+|I_2|)}} 
=C_{j,I_1',I_1'',I_2} \frac{1}{|z|^{2(2n+m-1+ \frac 12\la I\ra)}}.
\end{align*}
Since  $\langle I \rangle= |I_1'|+2|I_1''| + 2|I_2|$, this establishes the estimate in Theorem \ref{thm:pointwise bounds} for this term.
The remainder term \eqref{eqn:remainder N1} has a similarly straightforward adaptation and estimate.

\subsection{The case $0 < r \leq \frac 12$.} The estimates in this case will also follow from size estimates. We established the key estimate on $B(r,\nu)$
in \eqref{eqn:B-est}. Moreover, since the eigenvalues for $\nu\in S^{m-1}$ are bounded away from zero (say by $c$), we have
\[
\frac{1}{1-r^u} \leq \frac{1}{1-(1/2)^c} \leq C.
\]
It therefore follows that for $j=0,1,2$
\[
|\nabla_{z,\z}^j A(r,\nu,z)| \sim C|z|^{2-j}.
\]
Consequently, we ignore the $t$-term and estimate \eqref{eqn:NI1I2} directly by
\begin{align*}
|N_{I_1',I_1'',I_2} (z,t)|
&\leq C_{I_1',I_1'',I_2} \int_{r=0}^{\frac 12} \int_{\nu \in S^{m-1}} r^{c_0-1} 
\frac{|z|^{|I_1'|}d \nu \, dr}{|z|^{2(2n+m-1+|I_1'|+|I_1''|+|I_2|)}}  \\
&= \frac{C_{I_1',I_1'',I_2}}{|z|^{2(2n+m-1+ \frac 12\la I\ra)}}.
\end{align*}
This establishes the desired estimate in the case when $|z|^2 \geq |t|$ and hence concludes the proof of Theorem
\ref{thm:pointwise bounds}.
%
%
\section{The case $q=n$}\label{sec:q=n} 
The techniques that prove the estimates in the $q \neq n$ case 
are robust enough to work in the $q=n$ case, as well. However, the non-triviality of $\ker\Boxb$ changes 
for the formula for $N_K(z,t)$, and in this section, we sketch the argument for the $I=\emptyset$ case,  which is when there are no derivatives. We also assume, without
loss of generality, that the set of positive indices $P = \{1,2,\dots,n\}$.

We computed the relative solution to $\Box_b$ in the case $q=n$ given by $\int_0^\infty e^{-s\Boxb}(I-S_n)\, ds$ in \cite{BoRa21I}. 
Following the notation of \cite{BoRa21I}, for each $q$-tuple $L \in \I_q$, we set
\[
\Gamma_L = \{\alpha\in S^{m-1}: \mu^\alpha_\ell >0 \text{ for all }\ell\in L \text{ and } \mu^\alpha_\ell < 0 \text{ for all }\ell\not\in L\}.
\] 
If $L \in \I_n$, then $\Gamma_L = \emptyset$, unless $L=P$, in which case 
$\Gamma_P = S^{m-1}$. Therefore, from \cite[Theorem 2.2, Part 3]{BoRa21I},
if $K \in \I_n$, then
\begin{align}
\label{LP1}
&N_K(z,t) = K_{n,m} \bigg[ \sum_{L \in \I_n, L \not=P}  
\int_{\nu \in S^{m-1}} \det(\bar U(\nu)_{K,L}) \, d \bar Z(z, \nu)^L   \int_{r=0}^1
\frac{ B_L(r ,\nu)}{  ({A(r,\nu,z)}-i \nu \cdot t)^{2n+m-1}}  \frac{dr \, d \nu}{r} \\
&+ \int_{\nu \in S^{m-1}} \det (\bar U(\nu)_{K,P})\, d \bar Z(z, \nu)^P 
\int_{r=0}^1 \left[  \frac{B_P(r, \nu)}{  ({A(r,\nu,z)}-i \nu \cdot t)^{2n+m-1}}
- \frac{|\det A_\nu|}{(A(0, \nu, z)-i \nu \cdot t)^{2n+m-1}} \right] \frac{dr \, d \nu}{r}\bigg]
\label{LP2}
\end{align}

\subsection{The case $1/2 < r < 1$}
As above, we split up the $r$-integral into $0 < r \leq 1/2$ and $1/2 <r <1$.  
The challenge is in the region when $1/2 <r <1$
which is where we concentrate our efforts. In this case, the first fraction of the integrand in 
(\ref{LP2}) can be combined with the terms in (\ref{LP1}) so that the sum can range over all
$L \in I_n$ in (\ref{LP1}). 
After factoring out $|t|$ and rotating coordinates so that $\nu \cdot t/|t| = \nu_1$,
we set $\hatq=\frac{z}{|t|^{1/2}}$ and rewrite the integral over $1/2 \leq r \leq 1$ as follows:
\begin{align}
\label{NK1-first}
|t|^{2n+m-1} N_K^1 &= \sum_{L \in \I_n}  
\int_{\nu \in S^{m-1}} \det(\bar U(\nu_t)_{K,L}) \, d \bar Z(\hatq,{\nu_t})^L   \int_{r=1/2}^1
\frac{ B_L(r ,{\nu_t})}{  ({A(r,\nu_t,\hatq)}-i \nu_1)^{2n+m-1}}  \frac{dr \, d \nu}{r} \\
\label{NK1-second}
&- \int_{\nu \in S^{m-1}} \det(\bar U(\nu_t)_{K,P}) \, d \bar Z(\hatq, \nu_t)^P
\int_{r=1/2}^1 \frac{|\det A_{\nu_t}|}{(A(0, \nu_t, \hatq)-i \nu_1)^{2n+m-1}}  \frac{dr \, d \nu}{r} 
\end{align}
and where (as above)
$\nu_t=M_t^{-1} (\nu)$ and $M_t$ is an orthogonal transformation on $\R^m$ with  $M_t (t/|t|)=e_1$.
The analysis of (\ref{NK1-first}) is precisely the same as we carried out in Sections \ref{sec:notation}-\ref{sec:z large}.
Thus we focus on (\ref{NK1-second}). We show the following

\begin{prop}
\label{mainprop}
Let 
\begin{equation} 
\label{NK2}
N_K^2 (\hatq,t)=  \int_{\nu \in S^{m-1}} \det(\bar U(\nu_t)_{K,P}) \, d \bar Z(\hatq, \nu_t)^P
\int_{r=1/2}^1 \frac{|\det A_{\nu_t}|}{(A(0, \nu_t, \hatq)-i \nu_1)^{2n+m-1}}  \frac{dr \, d \nu}{r}
\end{equation}
Then there are positive constants $c_0$ and $C_0$ such that  $|N_K^2(\hatq,t)| \leq C_0$
for all $|\hatq| \leq c_0$.
\end{prop}
 
\begin{rem}Note that the case $|\hatq|>c_0$ falls into the $|z|$ dominant case which is much easier to handle.
\end{rem}
We devote the remainder of this section to proving this proposition.
To prove the proposition, we need the following analyticity lemma.
\begin{lem}
\label{analytic-lemma}
The following functions are analytic as a function of $\nu \in S^{m-1}$:

\begin{itemize} 
\item $\nu \to | \det (A_\nu)|$

\item $\nu \to A(0, \nu, \hatq)=\sum_{j=1}^{2n} |\mu_j^\nu||\hatq_j^\nu|^2$

\item $\nu \to 
\det(\bar U(\nu)_{K,P}) \, d \bar Z(\hatq, \nu)^P =
\sum_{J \in \I_n}  \det(\bar U(\nu)_{K,P}) \det [U(\nu)_{P,J}]^T \, d \bar z^J$

\end{itemize}

\end{lem}
\begin{proof}For the first bullet, note that $A_\nu$ has $n$  positive and $n$ negative 
eigenvalues and so $|\det A_\nu| = (-1)^n \det A_\nu$. So the expression in the first bullet is analytic since $A_\nu $ is linear in $\nu$. 

For the second bullet, note that 
$A_\nu$ and $|A_\nu| = \sqrt{A_\nu^2}$ have the same eigenvectors. Therefore
\begin{align}
\label{Anu formula}
\sum_{j=1}^{2n} |\mu_j^\nu||\hatq_j^\nu|^2 &= \sum_{j=1}^{2n} |\mu_j^\nu||Q_j(\nu,\hatq)|^2
= \hatq^* U(\nu) \cdot |D_\nu|\cdot  U(\nu)^*  \hatq = \hatq^* \sqrt{A_\nu^2} \, \hatq
\end{align}
which is analytic in $\nu$. Here, $D_\nu $ is the diagonal matrix with the eigenvalues of $A_\nu$ as its diagonal entriesl, and $\sqrt{ \  } $ is the principal branch of the square root of 
a positive definite Hermitian symmetric matrix.
\m

Showing the expression in the third bullet is analytic in $\nu$ is equivalent to showing that 
the expression
\begin{equation}
\label{eqn:U term for analyticity}
\det(\bar U(\nu)_{K,P}) \det[U(\nu)_{P,J}]^T
\end{equation} 
is real analytic in $\nu \in S^{m-1}$  for each $J,K $ in $\I_n$. 

We shall need  the standard branch of the function
$\arctan z$, which is holomorphic on $\C \setminus\{z = iy : x=0 \text{ and } |y|\geq 1\}$.
Let $I_{2n}$ be the $2n\times2n$ identity matrix and consider the sequence
\[
\frac 1\pi\Big( \arctan(nA^\nu)+\frac\pi2 I_{2n}\Big)
\]
for $j=1, 2, \dots $. Since the eigenvalues of $A^\nu$ are bounded away from zero,  each of these matrices in this sequence is analytic in $\nu \in S^{m-1}$ and
is diagonalized by $U(\nu)$ and $U(\nu)^*$. Furthermore, this sequence converges uniformly in $\nu$ as $j \to \infty$ to 
a matrix $A_0^\nu$, which is analytic in $\nu$ with $n$ eigenvalues equal to 1 and $n$ eigenvalues 
equal to $-1$.  Also $A_0^\nu$ is diagonalized by $U(\nu)$ and $[U(\nu)]^*$.

Now consider 
\[
\tilde A_0^\nu = \frac{1}{2} (A_0^\nu +I).
\]
$\tilde A_0^\nu$ is analytic in $\nu$ and has $n$ eigenvalues equal to 1 and $n$ eigenvalues equal to zero. It is also diagonalized by  $U(\nu)$ and $[U(\nu)]^*$. Therefore,

\[
[U(\nu)]^* \tilde A_0^\nu U(\nu) = D^0, \ \ \textrm{where} 
\]
\[D^0= \left( \begin{array}{cc}
I_n & 0_n \\
0_n & 0_n
 \end{array} \right)\] 
and where $I_n$ is the $n \times n$ identity matrix and $0_n$ is the $n \times n$ zero matrix. Therefore
\[
\overline{\tilde A_0^\nu} =\bar U(\nu) D^0 [U(\nu)]^T .
\]
Taking determinants, we have
\[
\det \overline{[\tilde A_0^\nu}]_{KJ} =\sum_{L, L'} \det (\bar U(\nu)_{K,L})
\det ( D^0_{L,L'} ) \det( [U(\nu)_{L',J}]^T) .
\]
Given that $D_0$ is diagonal, the only nonzero contributions to this sum occur when $L=L'=P$, which is
the set of positive indices $= 1, \dots, n$. We obtain
\[
\det \overline{[\tilde A_0^\nu}]_{KJ} = \det (\bar U(\nu)_{K,P})
 \det( [U(\nu)_{P,J}]^T) .
\]
Since the left side is analytic in $\nu$, so is the right side and this establishes the 
analyticity of (\ref{eqn:U term for analyticity}) and thus concludes the proof of the lemma.
\end{proof}

\begin{proof}[Proof of Proposition \ref{mainprop}] 
Note that the $r$-integral  in (\ref{NK2}) can be computed exactly (as $\ln(2)$), so we need only examine the $\nu$-integral.
Clearly, the integral over $|\nu_1| \geq 1/2$ clearly bounded uniformly in $q$ and $t$. So we restrict attention to
the region  $\{ \nu \in S^{m-1}; \ |\nu_1| \leq 1/2\}$. The key is to examine the integral over $\nu_1$-slices
of this region. Without loss of generality, let us assume we are on a region, $V$, of the sphere where $\nu_m =h(\nu_1, \nu')$ 
can be written as an analytic function of the other variables $(\nu_1, \nu')$ with $\nu'=(\nu_2, \dots , \nu_{m-1})$. We may further assume that
the ``cap" $V$  is large enough so that projection of $V$ onto $\nu_m=0$ contains the disk 
$\{(\nu_1,\nu') \in \R^{m-1}: |\nu_1|\leq \frac 12 \text{ and } |\nu'| \leq \frac 12\}$.
We also write $d \nu = g(\nu_1, \nu') d \nu_1 d \nu'$ where $g$ is an analytic function on $V$.
Let 
\[
G(\hatq,t, \nu_1, \nu')
= \det(\bar U_{K,P}^{\tilde \nu_t}) \, d \bar Z(\hatq, \tilde \nu_t)^P
 |\det A_{\tilde \nu_t}|\,  g(\nu_1, \nu')  \ \ \ \ \textrm{where $\tilde \nu_t=M_t^{-1} (\nu_1, \nu', h(\nu_1, \nu'))$}.
\]
From Lemma \ref{analytic-lemma}, $G(\hatq,t, \nu_1, \nu')$ is analytic in $\nu_1, \nu'$ and uniformly bounded in $\nu \in V \subset S^{m-1}$, $\hatq$, and $t$.
We need to show that there are positive constants $c_0$ and $C_0$ so that 
\begin{equation}
\label{goal-prop}
\Big| \int_{|\nu_1| \leq 1/2} \frac{G(\hatq,t, \nu_1, \nu') \, d \nu_1}{(A(0, \tilde \nu_t, \hatq)-i \nu_1)^{2n+m-1}}  \Big| \leq C_0
\quad\textrm{for all } |\nu'| \leq \frac 12 \text{ and }|\hatq| \leq c_0.
\end{equation}

We shall proceed by using Cauchy's Theorem to bump the contour of integration around the 
potential singularity at $\nu_1=0$.
First, let $\delta_0>0 $ be chosen small enough to that $G(\hatq,t, \nu_1, \nu')$ and $A(0,\tilde \nu_t, \hatq) $ analytically continue from $\{\nu_1 \in \R; \ |\nu_1| \leq 1/2 \}$ to a neighborhood of the rectangle $\tilde V_1 = \{ \zeta_1= \nu_1+i \eta_1 \in \C; \ |\nu_1| \leq 1/2 \ \textrm{and $0 \leq \eta_1 \leq \delta_0$} \}$ in the upper half plane and for all
$( \nu_1, \nu',h(\nu_1,\nu')) \in V$. 
Also note from (\ref{Anu formula}) that $A(0, \nu, \hatq) = \hatq^* \sqrt{A_\nu^2} \, \hatq \geq 0$ for $\nu \in V$. In addition, the analytic extension of $A(0, \nu, \hatq) $ to $\tilde V_1$ is the function
\[
A(0, \tilde \nu_t(\zeta_1), \hatq):= A(0, M_t^{-1} (\zeta_1, \nu', h(\zeta_1, \nu_t), \hatq).
\]
Furthermore, its $\zeta_1$ derivative is uniformly bounded by $\tilde C |\hatq|^2$
for $\zeta_1 \in \tilde V_1$ and $\nu \in V$ where $\tilde C>0$ is a uniform constant.
The following estimate now follows:
\[
\textrm{Re} \, A(0, \tilde \nu_t(\zeta_1), \hatq)
 \geq - \tilde C |\hatq|^2 \eta_1  \ \ 
\textrm{for $\zeta_1 = \nu_1 +i \eta_1 \in \tilde V_1$}.
\]
This inequality implies
\[
|A(0, \tilde \nu_t(\zeta_1),  \hatq)-i \zeta_1| \geq (1 -\tilde C |\hatq|^2) \eta_1 \ \ \ 
\textrm{for $\zeta_1 = \nu_1 +i \eta_1 \in \tilde V_1$}.
\]
Let $\gamma_1$ be the upper three sides of the boundary of the rectangle of $\tilde V_1$, i.e.
the union of the three line segments, respectively, from $-1/2$ to $-1/2 +i \delta$; from 
$-1/2 +i \delta$ to $1/2 +i \delta$, and from $1/2 +i \delta$ to $1/2$. The above inequality 
shows that there is a constant $c_0>0$  such that if $|\hatq| < c_0$, then
\begin{align*}
|A(0,\tilde \nu_t(\zeta_1), \hatq)-i \zeta_1| &>0  \ \ \textrm{for $\zeta_1 $ inside $\tilde V_1$ and} \\
|A(0,\tilde \nu_t(\zeta_1), \hatq)-i \zeta_1| & \geq  c_0 \ \ 
\ \ \textrm{for $\zeta_1 \in \gamma_1$ } .
\end{align*}
Now we can use Cauchy's Theorem to deform the path of integration in (\ref{goal-prop}) 
to $\gamma_1$ and the proof of the estimate in (\ref{goal-prop}) easily follows.
This concludes the proof of the proposition.
\end{proof}

\subsection{The cases $0 < r < 1/2$ and $|z|^2 >|t|$}
The estimates of $N_K$ for the interval $0 < r < 1/2$ 
follow the same arguments as given in Section \ref{sec:lower half}. 
The extra term arising from $S_n$ in \eqref{LP2} eliminates the
convergence issues at $r=0$.
Lemma \ref{analytic-lemma} and the earlier analyticity lemmas
show that all the components of the integrands are analytic in $\nu$. Since
$L \not=P$ in the sum in (\ref{LP1}), the numerator of its integrand contains a positive power of $r$. In addition, the term in brackets $[ \ ]$ in (\ref{LP2}) vanishes at $r=0$, so both integrands are integrable in $r$ near $r=0$. Therefore, the same integration by parts argument 
from Section \ref{sec:lower half} applies to reduce the power of the denominator terms
(down to a log-term) to prove the desired estimates. 

The case when $|z|^2 > |t|$ is handled using the techniques of Section \ref{z large}. 
This completes the proof of Theorem \ref{thm:pointwise bounds}.

%
%
\section{Examples}\label{sec:examples}
Here, we record four examples with complex tangent dimension $2n\geq 2$ and higher codimension $m \geq 2$ in cases where the eigenvalues are always nonzero. 
These examples piggy back on the following standard example in the case of $2n=2$ and $m=2$ 
originally computed in \cite{BoRa13q,BoRa20II}:
\m

\begin{example} $2n=2$, $m=2$, and $q=0$. Consider $\Phi(z,z) = ( \phi_1(z,z), \phi_2(z,z) )$
where
\begin{align*}
\phi_1(z,z) &= 2 \Rre  (z_1 \bar z_2) \\
\phi_2 (z,z) &= |z_1|^2 - |z_2|^2.
\end{align*}
The eigenvalues of the $A_\nu$ (the Hessian of $\Phi(z,z) \cdot \nu$) are $+1$ and $-1$
We use formula (\ref{eqn:N_K}) for $N_L$ with $L =\emptyset$ and so $\eps_1=-1$ and $\eps_2=+1$. Since $m=2$, $S^{m-1}$ is just the unit circle parameterized by $\nu = ( \cos \theta, \sin \theta) $, $0 \leq \theta \leq 2 \pi$ and $d \nu =d \theta$. We rotate $\theta $ coordinates so that $\nu \cdot t$ becomes $|t| \sin \theta$.  From (\ref{eqn:N_K}), we obtain
\[
N(z,t) =\frac{4^2}{2 (2 \pi)^4}  \int_0^1 \int_0^{2 \pi}
\frac{r}{(1-r)^2} \frac{2! \, d\theta }{\left[ \left( \frac{1+r}{1-r} \right) |z|^2 - i |t| \sin \theta
\right]^3} \frac{dr}{r}.
\]
We let $s=\frac{r+1}{1-r}$, $ds = \frac{2 \, dr}{(1-r)^2}$ to obtain
\begin{align*}
N(z,t) &=\frac{4^2}{2 (2 \pi)^4 } \int_0^{2 \pi} \int_1^\infty 
\frac{ds \, d \theta}{[s|z|^2-i |t|\sin \theta]^3} \\
&= \frac{4}{(2 \pi)^4} \frac{1}{|z|^2 |t|^2} \int_0^{2 \pi} \frac{d \theta}{[|\hatq|^2-i \sin \theta ]^2} 
\ \ \ \textrm{with $\hatq=z/|t|^{1/2}$} \\
&= \frac{1}{2 \pi^3} \frac{1}{[ |z|^4+ |t|^2]^{3/2}} \approx \frac{1}{(|z|+ |t|^{1/2})^6} \approx \frac{1}{\rho(z,t)^6}
\end{align*}
as indicated by Theorem \ref{thm:pointwise bounds}.

\end{example}

\begin{example} $2n=2$ and $m=3$. Consider $\Phi(z,z) = ( \phi_1(z,z), \phi_2(z,z), \phi_3 (z,z))$
where
\begin{align*}
\phi_1(z,z) &= 2 \Rre  (z_1 \bar z_2) \\
\phi_2 (z,z) &= |z_1|^2 - |z_2|^2 \\
\phi_3 (z,z) &= 2 \Imm   (z_1 \bar z_2).
\end{align*}
Let $\nu=(\nu_1, \nu_2, \nu_3) $ be a unit  vector  in $\R^3$. Then
\[
A_\nu = \left(
\begin{array}{cc}
\nu_2 & \nu_1 -i \nu_3 \\
\nu_1 +i \nu_3 & -\nu_2
\end{array}
\right).
\]
The characteristic equation for the eigenvalues is $\det(A_\nu-\lambda I) = \lambda^2 - |\nu|^2 = \lambda^2 -1$ with eigenvalues $\lambda = +1, \ -1$. 

From (\ref{eqn:N_K}) with $2n=2$ and $m=3$, we obtain
\[
N(z,t) =\frac{4^2 3!}{2 (2 \pi)^4}  \int_0^1 \int_{\nu \in S^2}
\frac{r}{(1-r)^2} \frac{d \nu \, dr }{\left[ \left( \frac{1+r}{1-r} \right) |z|^2 - i |t| \nu_1
\right]^4} \frac{dr}{r} .
\]
We now let $s=\frac{r+1}{r-1}$ as before and let $x=\nu_1$. Using (\ref{surfmeasure}), we write $d \nu = dx \, d \phi$  where $\phi $ is the angular measure of the $S^1$ copy of the equator of $S^2$. Since $\phi $ does not appear in the integrand, its integral provides a factor of $2\pi$. We obtain
\begin{align*}
N(z,t) &= \frac{4^2 3!}{4 (2 \pi)^4} \int_1^\infty \int_{x=-1}^1 \frac{2 \pi \, dx \, ds}{[s|z|^2-i|t|x]^4} \\
&= \frac{2}{\pi^3} \frac{1}{(|z|^4+|t|^2)^2} \approx \frac{1}{(|z|+ |t|^{1/2})^8} \approx \frac{1}{\rho(z,t)^8}
\end{align*}
as indicated by Theorem \ref{thm:pointwise bounds}.

\end{example}

\begin{example} $2n=4$, $m=4$, $q=0$. Consider $\Phi(z,z) = ( \phi_1(z,z), \phi_2(z,z),  \phi_3(z,z), \phi_4(z,z))$
where
\begin{align*}
\phi_1(z,z) &= 2  \textrm{Re} (z_1 \bar z_2) +2 \textrm{Re} (z_3 \bar z_4)\\
\phi_2 (z,z) &= 2 \textrm{Re} (z_2 \bar z_3) -2 \textrm{Re} (z_1 \bar z_4)\\
\phi_3(z,z) &= 2 \textrm{Im} (z_1 \bar z_2) -2 \textrm{Im} (z_3 \bar z_4)\\
\phi_4(z,z) &=- 2 \textrm{Im} (z_2 \bar z_3) +2  \textrm{Im} (z_1 \bar z_4).
\end{align*}
Let $\nu=(\nu_1, \nu_2, \nu_3, \nu_4) $ be a unit vector  in $\R^4$. Then
\begin{equation}\label{eqn:A_nu multiexample}
A_\nu =  \left(
\begin{array}{cccc}
0& \nu_1 -i\nu_3 & 0 & -\nu_2-i \nu_4 \\
 \nu_1 +i \nu_3 &0 &\nu_2 +i \nu_4 & 0 \\
0 & \nu_2-i\nu_4 &0 & \nu_1+i\nu_3 \\
-\nu_2 +i  \nu_4 &0&  \nu_1-i\nu_3 & 0
\end{array} 
\right).
\end{equation}
The characteristic polynomial (in $\lambda$) is the quadratic polynomial
$(\lambda^2-1)^2$ with eigenvalues $+1, +1, -1, -1$.

From (\ref{eqn:N_K}) with $2n=4$ and $m=4$, we obtain
\[
N(z,t) =\frac{4^4 6!}{2(2 \pi)^8 }\int_0^1 \int_{\nu \in S^3} \frac{r^2}{(1-r)^4}
\frac{d \nu}{\left[ \left( \frac{1+r}{1-r} \right) |z|^2 - i |t| \nu_1 \right]^7 } 
\frac{dr}{r}.
\]
We let $\hatq = z/|t|^{1/2}$ and $s= \frac{1+r}{1-r}$ (as before) to obtain
\[
N(z,t) =\frac{4^4 6!}{2^4(2 \pi)^8 |t|^7 }\int_1^\infty \int_{\nu \in S^3}
\frac{(s^2-1)\, d \nu \, ds}{(|\hatq|^2s-i \nu_1)^7}.
\]
Now we let $x=\nu_1$ and use (\ref{surfmeasure}) with $m=4$ to write
$d \nu = \sqrt{1-x^2} \, dx \, d \nu'$ where $d \nu' $ is surface measure on $S^2$ (the equator of $S^3$). When $\nu' $ is integrated out, this provides a factor of $4\pi$. We obtain
\begin{align*}
N(z,t) &=\frac{4^4 6! (4 \pi)}{2^4(2 \pi)^8 |t|^7 }\int_1^\infty \int_{x=-1}^1
\frac{(s^2-1) \sqrt{1-x^2} \, dx \, ds}{(|\hatq|^2s  - ix)^7} \\
&= \frac{15}{2 \pi^6}  \frac{1}{(|z|^4+|t|^2)^{7/2}} \approx \frac{1}{(|z|+|t|^{1/2})^{14}} \approx \frac{1}{\rho(z,t)^{14}}
\end{align*}
as indicated by Theorem \ref{thm:pointwise bounds}.
\end{example}

\begin{example}This is the same example as Example 12.3, except with $q=1$ and $K = \{1\}$. 
The matrix $A_\nu$ from \eqref{eqn:A_nu multiexample}
has associated eigensystem (with entries written as pairs $\{v,\lambda)\}$ where $v$ is a
unit eigenvector with
eigenvalue $\lambda$ is
\[
\left\{ \bigg\{ \frac{1}{\sqrt2}\begin{pmatrix} -1\\ -\nu_1-i\nu_3\\ 0 \\ \nu_2-i\nu_4  \end{pmatrix} , 1\bigg\},
\bigg\{ \frac{1}{\sqrt2} \begin{pmatrix}0 \\ \nu_2+i\nu_4 \\1\\ \nu_1-i\nu_3 \end{pmatrix}, 1\bigg\}
\bigg\{ \frac{1}{\sqrt2} \begin{pmatrix}1 \\ -\nu_1-i\nu_3\\0 \\ \nu_2-i\nu_4 \end{pmatrix}, -1\bigg\}
\bigg\{  \frac{1}{\sqrt2}\begin{pmatrix} 0\\ \nu_2+i\nu_4 \\ -1 \\ \nu_1-i\nu_3\end{pmatrix}, -1\bigg\}
\right\}.
\]
We have 
\[
U(\nu) =
\begin{pmatrix} v^\nu_1 & v^\nu_2 & v^\nu_3 & v^\nu_4 \end{pmatrix}
= \frac 1{\sqrt 2} \begin{pmatrix} -1 & 0 & 1 &0 \\
-\nu_1-i\nu_3 & \nu_2+i\nu_4 & -\nu_1-i\nu_3 & \nu_2+i\nu_4 \\
0 & 1 & 0 & -1\\
\nu_2-i\nu_4 & \nu_1-i\nu_3 & \nu_2-i\nu_4 & \nu_1-i\nu_3
\end{pmatrix}
\]
where
\[
v^\nu_1 = \frac{1}{\sqrt 2}\begin{pmatrix} -1\\ -\nu_1-i\nu_3\\ 0 \\ \nu_2-i\nu_4  \end{pmatrix}, \quad
v^\nu_2 = \frac{1}{\sqrt2} \begin{pmatrix}0 \\ \nu_2+i\nu_4 \\1\\ \nu_1-i\nu_3 \end{pmatrix}, \quad
v^\nu_3 =  \frac{1}{\sqrt 2}\begin{pmatrix}1 \\ -\nu_1-i\nu_3\\0 \\ \nu_2-i\nu_4 \end{pmatrix}, \quad
v^\nu_4 =  \frac{1}{\sqrt2}\begin{pmatrix} 0\\ \nu_2+i\nu_4 \\-1 \\ \nu_1-i\nu_3\end{pmatrix}
\]
so that
\[
U(\nu)^* A_\nu U(\nu) = \begin{pmatrix} I_2 & 0 \\ 0 & - I_2 \end{pmatrix}
\]
where $I_2$ is the $2\times2$ identity matrix. 

Since $|\mu_j^\nu|=1$ for all $j$ and $\nu \in S^{m-1}$, 
\begin{align*}
A_\nu (r,z) =  \sum_{j=1}^{2n} |\mu_j^\nu| \left( \frac{1+ r^{|\mu_j^\nu|}}{1-r^{|\mu_j^\nu|}} \right) |z^\nu_j|^2 
= \sum_{j=1}^4  \left( \frac{1+ r}{1-r} \right) |z^\nu_j|^2 = \frac{1+ r}{1-r} |z|^2.
\end{align*}
Next,
\begin{align*}
\prod_{j=1}^{2n} \frac{r^{(1/2)(1-\eps_{j,L}^\nu) |\mu_j^\nu|} |\mu_j^\nu|}
{(1-r^{|\mu_j^\nu|}) }
&= \frac{1}{(1-r)^4} \begin{cases} r & L \in \{1,2\} \\ r^3  & L \in \{3,4\} \end{cases}
\end{align*}
Next, we compute
\begin{align*}
d \bar Z(\nu,z) &= U(\nu)^T \cdot d \bar z
=  \frac 1{\sqrt 2}\begin{pmatrix}
-1 & -\nu_1-i\nu_3 & 0 &\nu_2-i\nu_4 \\ 
0 & \nu_2+i\nu_4 &1 & \nu_1-i\nu_3 \\
1 & -\nu_1-i\nu_3 &0 & \nu_2-i\nu_4 \\
0 & \nu_2+i\nu_4 & -1 & \nu_1-i\nu_3
\end{pmatrix} 
\begin{pmatrix} d\z_1 \\ d\z_2 \\ d\z_3 \\ d\z_4 \end{pmatrix} .
\end{align*}
Next,  $d \bar Z(\nu,z) = U(\nu)^T \cdot d \bar z$, then multiplying both sides by $\bar U(\nu)$ where 
$U(\nu)^T$ is the transpose of $U$ produces
\[
\bar U(\nu)\cdot d\bar Z(z,\nu) =\bar U(\nu) \cdot U(\nu)^T \cdot d \bar z = \overline{U(\nu) \cdot U(\nu)^*}\cdot d\bar z
= d\bar z.
\]
Also,  $\det(\bar U(\nu)_{K',L}) = \bar U(\nu)_{k',\ell}$.
From \eqref{eqn:N_K} and \eqref{eqn:K nm} and the computations in this example, 
\begin{align*}
& N_1(z,t) \\ &=-\frac{K_{4,4}}{2}  \int_{\nu \in S^3}
\Big[-d\z_1 - (\nu_1+i\nu_3)\, d\z_2  + (\nu_2-i\nu_4)d\z_4 \big)\Big]  
\int_{r=0}^1  \frac{r}{(1-r)^4} \frac{  \, d \nu}{(\frac{1+ r}{1-r} |z|^2- i \nu \cdot t)^7}\frac{dr}{r} \\
&+ \frac{K_{4,4}}{2}\int_{\nu \in S^3}\Big[d\z_1 - (\nu_1+i\nu_3)\,d\z_2  +( \nu_2-i\nu_4)\, d\z_4\Big] \int_{r=0}^1 
\frac{r^3}{(1-r)^4} \frac{  \, d \nu}{(\frac{1+ r}{1-r} |z|^2- i \nu \cdot t)^7}\frac{dr}{r} 
\end{align*}
where $K_{4,4} = \frac{4^4(6!)}{2(2 \pi)^8}$.
Reorganizing, we have
\begin{align}
\label{N1-example}
N_1(z,t) 
&=\bigg[ \frac{K_{4,4}}{2} \int_{\nu \in S^3} \int_{r=0}^1
\frac{1 + r^2}{(1-r)^4}
\frac{1}{(\frac{1+ r}{1-r} |z|^2- i \nu \cdot t)^7}\, dr\, d\nu \bigg]\, d\z_1 \\
&+\bigg[\frac{ K_{4,4}}{2} \int_{\nu \in S^3} \int_{r=0}^1
\frac{(\nu_1+i\nu_3) (1-r^2)}{(1-r)^4}
\frac{1}{(\frac{1+ r}{1-r} |z|^2- i \nu \cdot t)^7}\, dr\, d\nu \bigg]\, d\z_2 \nn \\
&-\bigg[ \frac{K_{4,4}}{2} \int_{\nu \in S^3} \int_{r=0}^1
\frac{(\nu_2-i\nu_4)(1-r^2)}{(1-r)^4}
\frac{1}{(\frac{1+ r}{1-r} |z|^2- i \nu \cdot t)^7}\, dr\, d\nu\bigg]\, d\z_4. \nn
\end{align}
We observe that with $\hatq = z/|t|^{1/2}$,
\[
\int_0^1 \frac{1+r^2}{(1-r)^4(\frac{1+r}{1-r} |\hatq|^2 + ia)^7}\, dr = \frac{-a^2 + 6 i a |\hatq|^2 + 25|\hatq|^4}{240|\hatq|^6(ia+|\hatq|^2)^6}
\]
and
\[
\int_0^1 \frac{1-r^2}{(1-r)^4(\frac{1+r}{1-r} |\hatq|^2 + ia)^7}\, dr = \frac{i a +6|\hatq|^2}{60|\hatq|^4(ia+|\hatq|^2)^6}.
\]

Let's also observe the estimate in the special case that $t = (|t|,0,\dots,0)$ and only the $d\z_1$ component
(since $K = \{1\}$) and compute
\begin{align*}
I &=   \int_{\nu \in S^3} \int_{r=0}^1
\frac{1+ r^2}{(1-r)^4}
\frac{1}{(\frac{1+ r}{1-r} |z|^2- i \nu \cdot t)^7}\, dr\, d\nu \\
&= \frac{1}{|t|^7} \int_{\nu \in S^3} \int_{r=0}^1
\frac{1+ r^2}{(1-r)^4}
\frac{1}{(\frac{1+ r}{1-r} |\hatq|^2- i \nu_1)^7}\, dr\, d\nu \\
&= \frac{1}{240|\hatq|^6|t|^7}\int_{\nu \in S^3}\frac{-\nu_1^2 - 6 i \nu_1 |\hatq|^2 + 25|\hatq|^4}{(|\hatq|^2-i\nu_1)^6}  \, d\nu.
\end{align*}
Integrating in spherical coordinates, we compute
\begin{align*}
I &= \frac{1}{240|t|^4|z|^6}\int_0^\pi \int_0^\pi \int_0^{2\pi}
\frac{-\cos^2\alpha_1 - 6 i \cos\alpha_1 |\hatq|^2 + 25|\hatq|^4}{(|\hatq|^2-i\cos\alpha_1)^6} 
 \sin^2 \alpha_1 \sin \alpha_2 \, d\alpha_3\, d\alpha_2\, d\alpha_1 \\
&= \frac{\pi}{120|t|^4|z|^6}\int_{-\pi}^{\pi}
\frac{-\cos^2\alpha_1 - 6 i \cos\alpha_1 |\hatq|^2 + 25|\hatq|^4}{(|\hatq|^2-i\cos\alpha_1)^6} 
\sin^2\alpha_1\, d\alpha_1.
\end{align*}
The last integral follows from the fact that $\cos\alpha_1$ and $\sin^2\alpha_1$ are even functions. If
$f(\cos\alpha_1,\sin\alpha_1)$ is the integrand, then
\begin{align*}
\int_{-\pi}^{\pi} f(\cos\alpha_1,\sin\alpha_1)\, d\alpha_1
= \oint_{|z|=1} f\Big(\frac{z+\frac 1z}2,\frac{z-\frac 1z}{2i}\Big) \frac{1}{iz}\, dz.
\end{align*}
If 
\[
g(z) =  f\Big(\frac{z+\frac 1z}2,\frac{z-\frac 1z}{2i}\Big) \frac{1}{iz}
\]
then 
\[
g(z) = -\frac{4iz(-1+z^2)^2(-100 |\hatq|^4 z^2 + (1 + z^2)^2 + 12 i |\hatq|^2 (z + z^3))}{(2 |\hatq|^2 z - i (1 + z^2))^6}
\]
has poles at $z =  i(-|\hatq|^2 \pm \sqrt{|\hatq|^4+1})$. The pole at $z=i(-|\hatq|^2+\sqrt{|\hatq|^4+1})$ occurs inside the unit disk and
is easily computed using Mathematica. In fact,
\[
\Res(g,i(-|\hatq|^2+\sqrt{|\hatq|^4+1})) =-\frac{5|\hatq|^6 i (-2 + 5 |\hatq|^4)}{2 (1 + |\hatq|^4)^{\frac92}}.
\]
In summary, the $d \bar z_1$ component of $N_1$ in (\ref{N1-example}) is
\[
I =  -2\pi i \frac{\pi}{120|t|^7}\frac{5 i (-2 + 5 |\hatq|^4)}{2 (1 + |\hatq|^4)^{\frac92}}.
\]

Similar calculations with $t = (|t|,0,0,0)$ shows that the $d \bar z_2$ component of $N_1$ in (\ref{N1-example}) is
\begin{align*}
II &= \int_{\nu \in S^3} \int_{r=0}^1 \frac{(\nu_1+i\nu_3) (1-r^2)}{(1-r)^4}
\frac{1}{(\frac{1+ r}{1-r} |z|^2- i \nu \cdot t)^7}\, dr\, d\nu
= \frac{2\pi^2 i}{15 |t|^7}\frac{35 |\hatq|^2}{8(1+|\hatq|^4)^{\frac 92}}.
\end{align*}
and the $d \bar z_4$ component of $N_1$ in (\ref{N1-example}) is
\begin{align*}
III &= \int_{\nu \in S^3} \int_{r=0}^1 \frac{(\nu_2-i\nu_4) (1-r^2)}{(1-r)^4}
\frac{1}{(\frac{1+ r}{1-r} |z|^2- i \nu \cdot t)^7}\, dr\, d\nu =0.
\end{align*}
In summary, we have computed that 
\[
N_1\big(z,(|t|,0,0,0)\big)
= \frac{15}{\pi^6}\frac{1}{(|t|^2+|z|^4)^{\frac 72}}\bigg[\frac{-2|t|^2 + 5|z|^4}{2(|t|^2+|z|^4)} \, d\z_1
+ i\frac{7|z|^2|t|}{(|t|^2+|z|^4)}\, d\z_2\bigg].
\]
\end{example}
This expression has norm $\approx \rho(z,t)^{-14} $ as indicated by Theorem \ref{thm:pointwise bounds}.

\begin{example} This is a modification of Example 12.3 where the eigenvalues of $A_\nu$ do not depend analytically on $\nu$.
Let $\Phi(z,z) = ( \phi_1(z,z), \phi_2(z,z),  \phi_3(z,z), \phi_4(z,z))$
where
\begin{align*}
\phi_1(z,z) &= 2  \textrm{Re} (z_1 \bar z_2) +2 \textrm{Re} (z_3 \bar z_4)\\
\phi_2 (z,z) &= 2 \textrm{Re} (z_2 \bar z_3) -2 \textrm{Re} (z_1 \bar z_4)\\
\phi_3(z,z) &= 2 \textrm{Im} (z_1 \bar z_2) -2 \textrm{Im} (z_3 \bar z_4)\\
\phi_4(z,z) &=- 2 \textrm{Im} (z_2 \bar z_3) +2 (1+b) \textrm{Im} (z_1 \bar z_4)
\end{align*}
where  $b$ is a small real number.
Let $\nu=(\nu_1, \nu_2, \nu_3, \nu_4) $ be a unit vector  in $\R^4$. Then, it is easy to compute the complex Hessian of $\Phi_\nu = \Phi(z,z) \cdot \nu$:
\[
A_\nu=\textrm{Hessian} \Phi_\nu = \left(
\begin{array}{cccc}
0& \nu_1 -i\nu_3 & 0 & -\nu_2-i (1+b)\nu_4 \\
 \nu_1 +i \nu_3 &0 &\nu_2 +i \nu_4 & 0 \\
0 & \nu_2-i\nu_4 &0 & \nu_1+i\nu_3 \\
-\nu_2 +i (1+b) \nu_4 &0&  \nu_1-i\nu_3 & 0
\end{array} 
\right).
\]
Note that when  $b=0$, then this is Example 3.

The characteristic polynomial (in $\lambda$) turns out to be a quadratic polynomial in $\lambda^2$ so the eigenvalues (though messy) can be computed as $\mu_1^\nu>0,\  \mu_2^\nu>0, \ 
\mu_3^\nu<0, \ \mu_4^\nu<0$ where:
\[
\mu_1^\nu = \sqrt{\Lambda_+} , \ \  \mu_2^\nu = \sqrt{\Lambda_-} , \ \ 
\mu_3^\nu =- \sqrt{\Lambda_+} , \ \ \mu_4^\nu = -\sqrt{\Lambda_-} 
\]
and where
\begin{align*}
\Lambda_{\pm} &=  \nu_1^2 +\nu_2^2 +\nu_3^2 +(1/2)(b^2+2b+2) \nu_4^2 \\
&\ \ \pm |\nu_4| |b| \left( \frac{(b+2)^2 \nu_4^2}{4} +  \nu_1^2 +\nu_3^2 \right)^{1/2}.
\end{align*}

Note that when $b=0$, $\mu_1^\nu =  \mu_2^\nu  =1$ and $\mu_3^\nu =\mu_4^\nu =-1$ 
as in Example 12.3. For  $ b$ nonzero, but small, these eigenvalues are not smooth at $\nu_4=0$ and $\nu_1, \nu_3 \not=0$ 
due to the presence of $|\nu_4|$.
\end{example}

\bibliographystyle{alpha}
\bibliography{mybib}

\begin{thebibliography}{NRSW89}

\bibitem[BGG96]{BeGaGr96}
R.~Beals, B.~Gaveau, and P.C. Greiner.
\newblock The {G}reen function of model step two hypoelliptic operators and the
  analysis of certain tangential {C}auchy {R}iemann complexes.
\newblock {\em Adv.\ Math.}, 121(2):288--345, 1996.

\bibitem[BGG00]{BeGaGr00}
R.~Beals, B.~Gaveau, and P.C. Greiner.
\newblock Hamilton-{J}acobi theory and the heat kernel on {H}eisenberg groups.
\newblock {\em J. Math. Pures Appl. (9)}, 79(7):633--689, 2000.

\bibitem[Bog91]{Bog91}
A.~Boggess.
\newblock {\em CR Manifolds and the Tangential Cauchy-Riemann Complex}.
\newblock Studies in Advanced Mathematics. CRC Press, Boca Raton, Florida,
  1991.

\bibitem[BRa]{BoRa21I}
A.~Boggess and A.~Raich.
\newblock The fundamental solution to {$\Box_b$} on quadric manifolds -- part
  1. general formulas.
\newblock {\em to appear, Proc. Amer. Math. Soc.}

\bibitem[BRb]{BoRa21III}
A.~Boggess and A.~Raich.
\newblock The fundamental solution to {$\Box_b$} on quadric manifolds -- part
  3. asymptotics for a codimension $2$ case in {$\mathbb{C}^4$}.
\newblock {\em submitted}.

\bibitem[BR09]{BoRa09}
A.\ Boggess and A.\ Raich.
\newblock A simplified calculation for the fundamental solution to the heat
  equation on the {H}eisenberg group.
\newblock {\em Proc.\ Amer.\ Math.\ Soc.}, 137(3):937--944, 2009.

\bibitem[BR11]{BoRa11}
A.\ Boggess and A.\ Raich.
\newblock The ${\Box}_b$-heat equation on quadric manifolds.
\newblock {\em J.\ Geom.\ Anal.}, 21:256--275, 2011.

\bibitem[BR13]{BoRa13q}
A.~Boggess and A.~Raich.
\newblock Fundamental solutions to {$\Box_b$} on certain quadrics.
\newblock {\em J.~Geom.~Anal.}, 23(4):1729--1752, 2013.

\bibitem[BR20]{BoRa20II}
A.~Boggess and A.~Raich.
\newblock The fundamental solution to {$\square_b$} on quadric manifolds:
  {P}art 2. {$L^p$} regularity and invariant normal forms.
\newblock {\em Complex Anal. Synerg.}, 6(2):Paper No. 13, 19, 2020.

\bibitem[BS17]{BiSt17}
S.~Biard and E.~Straube.
\newblock {$L^2$}-{S}obolev theory for the complex {G}reen operator.
\newblock {\em Internat. J. Math.}, 28(9):1740006, 31, 2017.

\bibitem[CCT06]{CaChTi06}
O.~Calin, D.-C. Chang, and J.~Tie.
\newblock Fundamental solutions for {H}ermite and subelliptic operators.
\newblock {\em J.\ Anal.\ Math.}, 100:223--248, 2006.

\bibitem[Chr91a]{Chr91a}
M.\ Christ.
\newblock On the $\bar\partial_b$ equation for three-dimensional {CR}
  manifolds.
\newblock In {\em Proceedings of Symposia in Pure Mathematics}, volume 52,
  {P}art 3, pages 63--82. American Mathematical Society, 1991.

\bibitem[Chr91b]{Chr91b}
M.\ Christ.
\newblock Precise analysis of $\bar\partial_b$ and $\bar\partial$ on domains of
  finite type in $\mathbb{C}^2$.
\newblock In {\em Proceedings of the {I}nternational {C}ongress of
  {M}athematicians, {V}ol.\ I, II ({K}yoto, 1990)}, pages 859--877, Tokyo,
  Japan, 1991. Math.\ Soc.\ Japan.

\bibitem[CR]{CoRa20}
J.~Coacalle and A.~Raich.
\newblock Closed range estimates for $\bar\partial_b$ on {CR} manifolds of
  hypersurface type.
\newblock {\em to appear, J. Geom. Anal.}

\bibitem[CS01]{ChSh01}
S.-C.\ Chen and M.-C.\ Shaw.
\newblock {\em Partial Differential Equations in Several Complex Variables},
  volume~19 of {\em Studies in Advanced Mathematics}.
\newblock American Mathematical Society, 2001.

\bibitem[Eld09]{Eld09}
N.~Eldredge.
\newblock Precise estimates for the subelliptic heat kernel on {$H$}-type
  groups.
\newblock {\em J. Math. Pures Appl. (9)}, 92(1):52--85, 2009.

\bibitem[FK88]{FeKo88a}
C.\ Fefferman and J.J.\ Kohn.
\newblock Estimates of kernels on three-dimensional {CR} manifolds.
\newblock {\em Rev.\ Mat.\ Iberoamericana}, 4(3-4):355--405, 1988.

\bibitem[FS74]{FoSt74p}
G.B.\ Folland and E.~Stein.
\newblock Parametrices and estimates for the {$\bar \partial _{b}$} complex on
  strongly pseudoconvex boundaries.
\newblock {\em Bull. Amer. Math. Soc.}, 80:253--258, 1974.

\bibitem[Gav77]{Gav77}
B.\ Gaveau.
\newblock Principe de moindre action, propogation de la chaleur, et
  estim{\'e}es sous elliptiques sur certains groupes nilpotents.
\newblock {\em Acta Math.}, 139:95--153, 1977.

\bibitem[HR11]{HaRa11}
P.~Harrington and A.~Raich.
\newblock Regularity results for $\bar\partial_b$ on {CR}-manifolds of
  hypersurface type.
\newblock {\em Comm.\ Partial Differential Equations}, 36(1):134--161, 2011.

\bibitem[HR15]{HaRa15}
P.~Harrington and A.~Raich.
\newblock Closed range for $\bar\partial$ and $\bar\partial_b$ on bounded
  hypersurfaces in {S}tein manifolds.
\newblock {\em Ann. Inst. Fourier (Grenoble)}, 65(4):1711--1754, 2015.

\bibitem[Hul76]{Hul76}
A.\ Hulanicki.
\newblock The distribution of energy in the {B}rownian motion in the {G}aussian
  field and analytic hypoellipticity of certain subelliptic operators on the
  {H}eisenberg group.
\newblock {\em Studia Math.}, 56:165--173, 1976.

\bibitem[Koh86]{Koh86}
J.J.\ Kohn.
\newblock The range of the tangential {C}auchy-{R}iemann operator.
\newblock {\em Duke Math.\ J.}, 53:525--545, 1986.

\bibitem[Mac88]{Mac88}
M.\ Machedon.
\newblock Estimates for the parametrix of the {K}ohn {L}aplacian on certain
  domains.
\newblock {\em Invent.\ Math.}, 91:339--364, 1988.

\bibitem[Men]{Men22}
G.~Mendoza.
\newblock Polarization set and {L}evi non-degeneracy.
\newblock {\em in preparation}.

\bibitem[NRS01]{NaRiSt01}
A.\ Nagel, F.\ Ricci, and E.M.\ Stein.
\newblock Singular integrals with flag kernels and analysis on quadratic {CR}
  manifolds.
\newblock {\em J.\ Funct.\ Anal.}, 181:29--118, 2001.

\bibitem[NRSW89]{NaRoStWa89}
A.\ Nagel, J.-P.\ Rosay, E.M.\ Stein, and S.\ Wainger.
\newblock Estimates for the {B}ergman and {S}zeg{\"o} kernels in ${{\mathbb
  C}}^2$.
\newblock {\em Ann.\ of Math.}, 129:113--149, 1989.

\bibitem[NS06]{NaSt06}
A.\ Nagel and E.M.\ Stein.
\newblock The $\dbar_b$-complex on decoupled domains in ${{\mathbb C}}^n$, $n
  \geq 3$.
\newblock {\em Ann.\ of Math.}, 164:649--713, 2006.

\bibitem[PR03]{PeRi03}
M.~Peloso and F.~Ricci.
\newblock Analysis of the {K}ohn {L}aplacian on quadratic {CR} manifolds.
\newblock {\em J.\ Funct.\ Anal.}, 2003(2):321--355, 2003.

\bibitem[Rai11]{Rai11}
A.~Rainer.
\newblock Quasianalytic multiparameter perturbation of polynomials and normal
  matrices.
\newblock {\em Trans. Amer. Math. Soc.}, 363(9):4945--4977, 2011.

\bibitem[Sha85]{Sha85}
M.-C.\ Shaw.
\newblock ${L}\sp 2$-estimates and existence theorems for the tangential
  {C}auchy-{R}iemann complex.
\newblock {\em Invent.\ Math.}, 82:133--150, 1985.

\bibitem[Sta99]{Sta99}
R.~P. Stanley.
\newblock {\em Enumerative combinatorics. {V}ol. 2}, volume~62 of {\em
  Cambridge Studies in Advanced Mathematics}.
\newblock Cambridge University Press, Cambridge, 1999.

\bibitem[YZ08]{YaZh08}
Q.~Yang and F.~Zhu.
\newblock The heat kernel on {H}-type groups.
\newblock {\em Proc.\ Amer.\ Math.\ Soc.}, 136(4):1457--1464, 2008.

\end{thebibliography}
\end{document}